\documentclass{amsart} 

\usepackage{amsmath, amssymb, amsfonts, tikz, hyperref, comment, url, todonotes,mathtools,enumitem}
\usepackage{bbm}
\usepackage{pgfplots}

\usetikzlibrary{math}
\usetikzlibrary{decorations.pathreplacing}

\makeatletter
\newtheorem*{rep@theorem}{\rep@title}
\newcommand{\newreptheorem}[2]{%
\newenvironment{rep#1}[1]{%
 \def\rep@title{#2 \ref{##1}}%
 \begin{rep@theorem}}%
 {\end{rep@theorem}}}
\makeatother

\newtheorem{theorem}{Theorem}[section]
\newreptheorem{theorem}{Theorem}
\newtheorem{lemma}[theorem]{Lemma}
\newtheorem{proposition}[theorem]{Proposition}

\newtheorem{observation}[theorem]{Observation}
\theoremstyle{definition}

\newtheorem{remark}[theorem]{Remark}
\newtheorem{definition}[theorem]{Definition}
\newtheorem*{definition*}{Definition}
\newtheorem*{theorem*}{Theorem}
\newtheorem*{proposition*}{Proposition}
\newtheorem*{corollary*}{Corollary}

\newtheorem{conjecture}[theorem]{Conjecture}

\newcommand{\Var}{\mathrm{Var}}
\newcommand{\Cov}{\mathrm{Cov}}

\author{Kaarel H\" anni}
\address{Department of Mathematics, Massachusetts Institute of Technology, \mbox{Cambridge, MA 02139}}
\email{\href{mailto:kaarelh@mit.edu}{{\tt kaarelh@mit.edu}}}

\begin{document}

\title{Asymptotics of descent functions}
\date{August 2020}

\begin{abstract}
In 1916, MacMahon showed that permutations in $S_n$ with a fixed descent set $I$ are enumerated by a polynomial $d_I(n)$. Diaz-Lopez, Harris, Insko, Omar, and Sagan recently revived interest in this \emph{descent polynomial}, and suggested the direction of studying such enumerative questions for other consecutive patterns (descents being the consecutive pattern $21$). Zhu studied this question for the consecutive pattern $321$. We continue this line of work by studying the case of any consecutive pattern of the form $k,k-1,\ldots,1$, which we call a \emph{$k$-descent}. In this paper, we reduce the problem of determining the asymptotic number of permutations with a certain $k$-descent set to computing an explicit integral. We also prove an equidistribution theorem, showing that any two sparse $k$-descent sets are equally likely.

Counting the number of $k$-descent-avoiding permutations while conditioning on the length $n$ and first element $m$ simultaneously, one obtains a number triangle $f_k(m,n)$ with some useful properties. For $k=3$, the $m=1$ and $m=n$ diagonals are OEIS sequences A049774 and A080635. We prove a $k$th difference recurrence relation for entries of this number triangle. This also leads to an $O(n^2)$ algorithm for computing $k$-descent functions. 

Along the way to these results, we prove an explicit formula for the distribution of first elements of $k$-descent-avoiding permutations, as well as for the joint distribution of first and last elements. We also develop an understanding of discrete order statistics. In our approach, we combine algebraic, analytic, and probabilistic tools. A number of open problems are stated at the end.

\end{abstract}

\maketitle

\section{Introduction}


A permutation $w\in S_n$ is said to \emph{contain} the consecutive pattern $\pi\in S_k$ if there are consecutive indices $i,i+1,\ldots, i+k-1$ such that the relative ordering of $w(i),w(i+1),\ldots, w(i+k-1)$ is the same as the relative ordering of $\pi(1), \ldots, \pi(k)$. A permutation $w$ is said to \emph{avoid} the consecutive pattern $\pi$ if it does not contain $\pi$. The study of consecutive pattern avoidance was started by Elizalde and Noy in 2003 \cite{ELIZALDE2003110}, and has received a great amount of study since. In this paper, we are interested in a slightly different topic, namely the study of permutations containing a consecutive pattern at some fixed set of indices, continuing a line of inquiry started by MacMahon in 1916 for the case of descents \cite{macmahon}. In particular, we will be interested in the consecutive pattern  $k,k-1, \ldots ,2,1$, which we call a \emph{k-descent}. For a permutation $w\in S_n$ and $k\geq 2$, we let $D_k(w)$ be the set of starting points of $k$-descents in $w$. By a starting point of a $k$-descent in $w$, we mean an index $i\in [n]$ such that $w(i)>w(i+1)>\cdots>w(i+k-1)$. For instance, for $w=638541972\in S_9$, the set of starting points of $k$-descents in $w$ is $D_k(w)=\{3,4,7\}$. Our main objects of interest are defined as follows.
\begin{definition}
For $n\in \mathbb{Z}^+$ and $I\subseteq \mathbb{Z}^+$ a finite set, we let
\[\mathcal{D}_k(I,n)=\{w\in S_n\colon D_k(w)=I\}\hspace{5mm}\text{and}\hspace{5mm} d_k(I,n)=|\mathcal{D}_k(I,n)|.\]
We call $d_k(I,n)$ the \emph{$k$-descent function}. 
\end{definition}

The case of $k=2$ has received a considerable amount of interest. MacMahon \cite{macmahon} proved that for fixed $I$, $d_k(I,n)$ is a polynomial in $n$ (for all $n$ sufficiently large); the function $d_k(I,n)$ is known as the \emph{descent polynomial}. Inspired by the work of Billey, Burdzy, and Sagan \cite{billey} on the adjacent topic of peak polynomials, which has received a large amount of further study \cite{MR3463566}\cite{DAVIS20183249}\cite{DIAZLOPEZ201721}, Diaz-Lopez, Harris, Insko, Omar, and Sagan \cite{diaz} recently revived interest in this descent polynomial. Their 2019 paper led to a number of other recent works on descent polynomials \cite{gaetz2019qanalogs}\cite{jiradilok2019roots}\cite{Oguz2019DescentPP}. They also suggested the direction of studying similar questions for other consecutive patterns (Section 6 part (1) \cite{diaz}). Zhu picked up this study \cite{zhu2019enumerating} for the case of the consecutive pattern $321$, i.e. $k=3$ in our notation.

In this paper, we will focus on the study of asymptotics of $d_k(I,n)$ for $k\geq 3$. For our purposes, it will turn out to be particularly useful to partition the set $\mathcal{D}_k(I,n)$ according to the value of the first element of the permutation.
\begin{definition} For $m,n\in \mathbb{Z}^+$ with $1\leq m\leq n$ and $I\subseteq \mathbb{Z}^+$ a finite set, we let
\[\mathcal{D}_k(I,m,n)=\{w\in S_n\colon D_k(w)=I \text{ and } w(1)=m\}\hspace{5mm}\text{and}\hspace{5mm} d_k(I,m,n)=|\mathcal{D}_k(I,m,n)|.\]
We call $d_k(I,m,n)$ the \emph{parametrized $k$-descent function}. 
\end{definition}

We will particularly care about the special case of $I=\emptyset$, which is exactly the case of consecutive pattern avoidance. We introduce the following shorthand notations to avoid notational clutter.
\begin{definition}
For $m,n\in \mathbb{Z}^+$ with $1\leq m\leq n$, we let
\[f_k(n)=d_k(\emptyset, n)\hspace{5mm} \text{and}\hspace{5mm} f_k(m,n)=d_k(\emptyset, m,n).\]
\end{definition}

We now give an outline of our paper, stating our main results. In Section~\ref{sec:rec}, we prove a recurrence relation for $f_k(m,n)$ (Theorem~\ref{thm:mainrec}), which gives rise to a fast algorithm (thinking of $k$ as fixed, and $n,m$ as parameters) for computing $f_k(m,n)$. Along similar lines, we present a fast algorithm for computing $d_k(I,n)$ for any fixed $I$ (Theorem~\ref{thm:otherrec}). We also give a bivariate generating function for $f_3(m,n)$ (Proposition~\ref{prop:gen}), and discuss its generalization to other $k$. In Section~\ref{sec:nasy}, we review some results from the consecutive pattern avoidance literature on the asymptotics of $f_k(n)$. In Section~\ref{sec:mnasy}, we derive the asymptotics of $f_k(m,n)$, with the primary motivation being that this will be crucial in proving our other main results. However, this can also be seen as a statement about the distribution of the first element statistic among permutations avoiding $k$-descents, which in our opinion can be an interesting result in its own right. Perhaps surprisingly, in the following main theorem of Section~\ref{sec:mnasy}, we see that asymptotically, the distribution of $f_k(m,n)$ approaches an explicit smooth distribution $\varphi_k$ (under the right normalization).

\begin{reptheorem}{thm:fmnasy}
For all $k\geq 3$, there is a constant $r_k$ so that for all $m,n\in \mathbb{Z}^+$ with $1\leq m\leq n$,
\[\frac{nf_k(m,n)}{f_k(n)}=\varphi_k\left(\frac{m}{n}\right)\left(1+O_k\left(n^{-0.49}\right)\right),\]
where 
\[\varphi_k\left(x\right)=\frac{1}{r_k}\left(1-\frac{(x/r_k)^{k-1}}{(k-1)!}+\frac{(x/r_k)^k}{k!}-\frac{(x/r_k)^{2k-1}}{(2k-1)!}+\frac{(x/r_k)^{2k}}{(2k)!}-\cdots \right).\]
\end{reptheorem}

Using this, in Section~\ref{sec:dasy} we prove the following theorem on the asymptotics of $d_k(I,n)$ as a corollary of some more precise asymptotic results (Theorem~\ref{thm:precdasy} or Proposition~\ref{prop:niceform}):

\begin{reptheorem}{thm:dasy}
For any $k\geq 3$ and finite $I\subseteq \mathbb{Z}^+$, there is a constant $c_{I,k}$ such that
\[d_k(I,n)=c_{I,k} f_k(n)\left(1+O(n^{-0.49})\right).\]
\end{reptheorem}
In fact, the constant $c_{I,k}$ can be computed (or bounded) efficiently, as it is given by a certain integral formula. This directly settles a conjecture (Conjecture 6.5 \cite{zhu2019enumerating}) by Zhu, and lets us make partial progress towards Zhu's Down-Up-Down-Up Conjecture (Conjecture 6.2 \cite{zhu2019enumerating}). To summarize our approach to proving the above theorem in a few words, the two main ideas are that (1) counting permutations with a certain property is equivalent to finding the probability that a random permutation has a certain property, and that (2) for a certain property, this probability should approach a constant as $n\to \infty$. In Section~\ref{sec:joint}, we bootstrap from the results of Section~\ref{sec:mnasy} to get a description of the joint distribution of the first and last element of a $k$-descent avoiding permutation. Namely, the first and last element turn out to be (almost) independent -- see Theorem~\ref{thm:joint} for a precise statement. Finally, in Section~\ref{sec:equi}, we will use this joint distribution result to prove the following equidistribution theorem.

\begin{reptheorem}{thm:equi}
Fix $k\geq 3$ and $r\in \mathbb{Z}^+$. Let $n\in \mathbb{Z}^+$, $I_1,I_2\subseteq [n]$ with $|I_1|=|I_2|=r$, and no two elements of $I_1$ being closer to each other than $\sqrt{n}$, and similarly for $I_2$. Then 
\[\frac{d_k(I_1,n)}{d_k(I_2,n)}=1+O_{k,r}\left(n^{-\alpha}\right).\]
\end{reptheorem}

Restated another way, the content of the above theorem is that for any two sparse enough $k$-descent sets $I_1,I_2$ of the same size, the number of permutations in $S_n$ with descent set $I_1$ is (almost) the same as the number of permutations with descent set $I_2$. This resolves a conjecture (Conjecture 6.1 \cite{zhu2019enumerating}) by Zhu, stated for the special case of $k=3$ and singleton $I$. On the way to proving this, we use binomial coefficient sum manipulation and the second moment method to prove a concentration result for discrete order statistics.

We finish this introduction with a remark on a simple extension of our results.
\begin{remark}
Taking complements (the complement of $w\in S_n$ is $w^c\in S_n$ defined by $w^c(i)=n+1-w(i)$), one obtains results analogous to Theorem~\ref{thm:fmnasy}, Theorem~\ref{thm:dasy}, Theorem~\ref{thm:joint}, and Theorem~\ref{thm:equi} for the consecutive pattern $1,2,\ldots, k$.
\end{remark} 

\section{A recurrence relation for descent functions}\label{sec:rec}

We begin by giving an outline of this section. In Subsection~\ref{subsec:simprec}, we state and prove Theorem~\ref{thm:mainrec}, giving a simple recurrence relation for $f_k(m,n)$. Next, in Subsection~\ref{subsec:fast}, we discuss how this recurrence allows fast computation of $f_k(m,n)$. In Subsection~\ref{subsec:heuristic}, we take some time off to have a strictly heuristic discussion of what one would expect for the distribution of $f_k(m,n)$ just from the recurrence in Theorem~\ref{thm:mainrec}.

We come back to the rigorous path in Subsection~\ref{subsec:general}, where we state and prove Theorem~\ref{thm:otherrec}, which is a generalization of Theorem~\ref{thm:mainrec} for $\mathcal{D}_k(I,m,n)$ with any $I$ (Theorem~\ref{thm:mainrec} is the case of $I=\emptyset$). As before, we show how this allows fast computation of $\mathcal{D}_k(I,m,n)$. We finish this section with some discussion of generating functions in Subsection~\ref{subsec:generating}.

\subsection{A recurrence relation for $f_k(m,n)$}\label{subsec:simprec}
We start by defining a familiar function.

\begin{definition}
Let $\mathcal{X}$ be the set of all finite length sequences of reals (including the empty sequence). We define the \emph{difference operator} $\Delta\colon \mathcal{X}\to \mathcal{X}$. For $n\geq 2$, $\Delta$ is given by \[\Delta((a_1,a_2,\ldots, a_n))=(a_1-a_2, a_2-a_3, \ldots, a_{n-1}-a_n),\]
and we adopt the convention that for $n=1$ and $n=0$, $\Delta(A)=(\hspace{1mm})$, the empty sequence.

We say that the \emph{$k$th difference of $A$} is $\Delta^k(A)$, i.e., the $k$th iterate of the function $\Delta$ applied to $A$.
\end{definition}

We define multiplication of sequences by constants and addition of sequences of the same length componentwise, i.e., like vectors. It will be useful to note for later that the $k$th difference is linear:
\begin{itemize}
    \item for any $A\in \mathcal{X}$ and scalar $c\in \mathbb{Z}$, $\Delta^k(cA)=c\Delta^k(A)$;
    \item for any $A,B\in \mathcal{X}$ of the same length, $\Delta^k(A+B)=\Delta^k(A)+\Delta^k(B)$.
\end{itemize}

We now state the main recurrence theorem.
\begin{theorem}\label{thm:mainrec}
For any integers $k\geq 2$ and $n\geq 1$,
\[\Delta^k\left(\left(f_k(1,n+k),f_k(2,n+k),\ldots, f_k(n+k,n+k)\right)\right)=\left(f_k(1,n),f_k(2,n),\ldots,f_k(n,n)\right).\]
\end{theorem}

For the proof, we start from the following more complicated recursive formula.
\begin{proposition}\label{prop:sum}
For any $k\geq 2$ and any integers $m,n\in \mathbb{Z}^+$ with $1\leq m\leq n+k$,
\[f_k(m,n+k)=f_k(n+k-1)-\left(\sum_{u=1}^{n}f_k(u,n)\binom{m-1}{k-1}-\sum_{u=1}^{\min(m-1,n)}f_k(u,n)\binom{m-1-u}{k-1}\right).\]
\end{proposition}

\begin{proof}
For $w\in \mathcal{D}_k(m,n+k)$, the first element of $w$ is $m$, and the restriction of $w$ to the last $n+k-1$ indices is an element of $\mathcal{D}_k\left(\emptyset, n+k-1\right)$. However, not every element of $\mathcal{D}_k(m,n+k-1)$ can be inserted here -- the elements that cannot be inserted are precisely those that start with a decreasing sequence of length $k-1$ starting from $m'<m$. Writing this out explicitly, we have that $f_k(m,n+k)=f_k(n+k-1)-|\mathcal{A}|$, where
\[\mathcal{A}=\{v\in \mathcal{D}_k\left(\emptyset, n+k-1\right)\colon m>v(1)>v(2)>\cdots >v(k-1)\}.\]

Fixing some $u\in [n]$, we now consider the number of elements $v\in \mathcal{A}$ such that the restriction of $v$ to the last $n$ indices starts with $u$. We can construct any such element uniquely by choosing the restriction to the last $n$ indices, for which there are $f_k(u,n)$ options, and then choosing values for the initial decreasing subsequence of length $k-1$. If $u\geq m$, then any set of values strictly less than $m$ will be a suitable choice of values for the initial decreasing subsequence of length $k-1$, so the number of options for the subsequence is $\binom{m-1}{k-1}$. If $u<m$, then the number of options for values of the decreasing subsequence in $v$ such that the value of the first element in $w$ is less than $m$ is still $\binom{m-1}{k-1}$, but out of those options, exactly $\binom{m-u-1}{k-1}$ give a decreasing subsequence of length $k$ in $v$ as well (here we use the fact that $v(k-1)>v(k)\iff v(k-1)\geq u$), and hence will give a permutation not in $\mathcal{A}$. Putting these cases together and summing over $u$, we get
\[|\mathcal{A}|=\sum_{u=1}^{n}f_k(u,n)\binom{m-1}{k-1}-\sum_{u=1}^{\min(m-1,n)}f_k(u,n)\binom{m-1-u}{k-1},\]
from which
\[f_k(m,n+k)=f_k(n+k-1)-\left(\sum_{u=1}^{n}f_k(u,n)\binom{m-1}{k-1}-\sum_{u=1}^{\min(m-1,n)}f_k(u,n)\binom{m-1-u}{k-1}\right),\]
which is the desired formula.
\end{proof}


Theorem~\ref{thm:mainrec} follows with some sum manipulation.
\begin{proof}[Proof of Theorem~\ref{thm:mainrec}]
By Proposition~\ref{prop:sum} and linearity of $\Delta^k$, 

\begin{align*}
    \Delta^k\left((f_k(m,n+k))_{m=1,\ldots,n+k}\right)&=\Delta^k\left((f_k(n+k-1))_{m=1,\ldots,n+k}\right)\\
    &\hphantom{=}-\sum_{u=1}^{n}f_k(u,n)\Delta_k\left(\left(\binom{m-1}{k-1}\right)_{m=1,\ldots,n+k}\right)\\
    &\hphantom{=}+\Delta^k\left(\left(\sum_{u=1}^{\min(m-1,n)}f_k(u,n)\binom{m-1-u}{k-1}\right)_{m=1,\ldots,n+k}\right).
\end{align*}
The first term is a constant sequence, so its $k$th difference is the all $0$s sequence $\mathbf{0}$. The second term is $\mathbf{0}$ as well, since \[\left(\binom{m-1}{k-1}\right)_{m=1,\ldots n+k}\xrightarrow{\Delta}\left(\binom{m-1}{k-1}\right)_{m=1,\ldots n+k-1}\xrightarrow{\Delta}\cdots \xrightarrow{\Delta}\left(\binom{m-1}{0}\right)_{m=1,\ldots n+1}\xrightarrow{\Delta} \mathbf{0}.\]
Here, we used the binomial coefficient identity $\binom{m}{k-1}=\binom{m-1}{k-1}+\binom{m-1}{k-2}$. For the third (final) sum, a similar thing happens for all terms except the last summand. We start by finding the first difference:
\[\Delta\left(\left(\sum_{u=1}^{\min(m-1,n)}f_k(u,n)\binom{m-1-u}{k-1}\right)_{m=1,\ldots,n+k}\right)\]\[=\left(\left(\sum_{u=1}^{\min(m-1,n)}f_k(u,n)\binom{m-1-u}{k-2}\right)+\mathbbm{1}_{m\leq n}f_k(m,n)\binom{0}{k-1}\right)_{m=1,\ldots,n+k-1}.\]
Note that $\binom{0}{k-1}=0$, so we are left with just the sum term. As we take successive differences (formally by induction), this pattern continues: after taking $\ell<k$ finite differences, we get the same sum with $k-1-\ell$ replacing $k-1$ in the binomial coefficients, as well as a leftover term of $\binom{0}{k-\ell}$, which is $0$. In particular, after doing this $k-1$ times, we have $k-1$ replaced by $0$ in the binomial coefficients, which means that the matching index terms cancel with the next $\Delta$. The only nonzero contribution comes from the leftover term of $f_k(m,n)\binom{0}{0}=f_k(m,n)$. Hence, we arrive at the desired
\[\Delta^k\left((f_k(m,n+k))_{m=1,\ldots,n+k}\right)=\left(f_k(m,n)\right)_{m=1,\ldots,n}.\]
\end{proof}

\subsection{Fast computation of $f_k(m,n)$}\label{subsec:fast}
Theorem~\ref{thm:mainrec} lets us compute the values of $f_k(m,n)$ for all $m,n$ with $1\leq m\leq n\leq N$, in $O(k N^2)$ arithmetic operations. That is, we are computing $\Theta(N^2)$ numbers in $O_k(N^2)$ operations. We think of these $f_k(m,n)$ as forming a triangle of integers, with the first row containing $f_k(1,1)$, the second containing $f_k(1,2), f_k(2,2)$, and so on. See Figure~\ref{fig:triangle} for a picture of this layout, and Figure~\ref{fig:fmn} for this triangle for $k=3$. The first $k$ rows are easy to find; namely, for any $1\leq m\leq n\leq k$ except $n=m=k$, $f_k(m,n)=(n-1)!$, and $f_k(k,k)=(k-1)!-1$ (so for $n\leq k$, each row can be computed with one multiplication from the previous row). To find subsequent rows one by one, we use the recurrence from Theorem~\ref{thm:mainrec}. Namely, when finding the $(n+k)$th row, we look at the $n$th row, which forms its $k$th difference. We now compute $k$ antidifferences (inverse finite differences) successively. We think of this as first laying out (in a new triangle) the $n$th row, which is the $k$th difference. We now find the $(k-1)$th difference in the row below it, by noting that the last entry in the $(k-1)$th difference sequence is $0$ (one can observe this in the proof of Theorem~\ref{thm:mainrec}, noting that the sums cancel), and then filling out the rest of the terms using the row above which is its first difference. For instance, the $n$th element of the second row is the last element of the second row minus the $n$th entry in the row above; the $(n-1)$th element of the second row is the $n$th element of the second row minus the $(n-1)$th entry in the row above; and so on, filling the second row backwards. We then repeat for the third row, using the fact that the second row is its first difference, except now we start from the first term being $0$ -- this follows from the fact that $f_k(1,n+k)=f_k(2,n+k)=\cdots=f_k(k-1,n+k)=f(n+k-1)$. For the fourth row, the first element is again $0$, and we fill it out using the third row as its first difference. In fact, the first element is $0$ for the third up to $k$th row, and we fill these out one by one. Each time, we only need $\leq n+k$ additions to find the terms in a row. For the $(k+1)$th row, we have that the first element is $f_k(1,n+k)=f_k(n+k-1)$ (this is easy to see combinatorially, by noting that if the first element is $1$, the restriction to the last $n+k-1$ elements can be any permutation in $\mathcal{D}_k(\emptyset,n+k-1)$) which can be computed as the sum of the elements of the previous row of our triangle, but the rest of the procedure is the same as for previous rows. All in all, given previous rows of the triangle, we have computed the $f_k(1,n+k),f_k(2,n+k),\ldots, f_k(n+k,n+k)$ row of our triangle in $\Theta(nk)$ additional arithmetic operations. For the case $k=3$, this process is depicted in Figure~\ref{fig:comp}. 

\begin{figure}
    \centering
    \begin{tikzpicture}[scale=1.5]
\foreach \n in {1,...,4} {
  \foreach \k in {1,...,\n} {
    \node at (\k-\n/2,-0.7*\n) {$f_k(\k,\n)$};
  }
}
\end{tikzpicture}
    \caption{It can be convenient to think of $f_k(m,n)$ as occupying such a triangle.}
    \label{fig:triangle}
\end{figure}

\begin{figure}
    \centering
    \begin{tikzpicture}
    
    \node at (1-1/2,-0.7*1) {$1$};
    \node at (1-2/2,-0.7*2) {$1$};
    \node at (2-2/2,-0.7*2) {$1$};
    \node at (1-3/2,-0.7*3) {$2$};
    \node at (2-3/2,-0.7*3) {$2$};
    \node at (3-3/2,-0.7*3) {$1$};
    \node at (1-4/2,-0.7*4) {$5$};
    \node at (2-4/2,-0.7*4) {$5$};
    \node at (3-4/2,-0.7*4) {$4$};
    \node at (4-4/2,-0.7*4) {$3$};
    \node at (1-5/2,-0.7*5) {$17$};
    \node at (2-5/2,-0.7*5) {$17$};
    \node at (3-5/2,-0.7*5) {$15$};
    \node at (4-5/2,-0.7*5) {$12$};
    \node at (5-5/2,-0.7*5) {$9$};
    \node at (1-6/2,-0.7*6) {$70$};
    \node at (2-6/2,-0.7*6) {$70$};
    \node at (3-6/2,-0.7*6) {$65$};
    \node at (4-6/2,-0.7*6) {$57$};
    \node at (5-6/2,-0.7*6) {$48$};
    \node at (6-6/2,-0.7*6) {$39$};
    \node at (1-7/2,-0.7*7) {$349$};
    \node at (2-7/2,-0.7*7) {$349$};
    \node at (3-7/2,-0.7*7) {$332$};
    \node at (4-7/2,-0.7*7) {$303$};
    \node at (5-7/2,-0.7*7) {$267$};
    \node at (6-7/2,-0.7*7) {$228$};
    \node at (7-7/2,-0.7*7) {$189$};
    
\end{tikzpicture}
    \caption{The triangle from Figure~\ref{fig:triangle} for $k=3$}
    \label{fig:fmn}
\end{figure}

\begin{figure}
    \centering
    \begin{tikzpicture}
    
    \node at (1-8/2,-0.7*5+0.02) {\color{green} $n$th row:};
    \node at (2-12/2,-0.7*6+0.02) {\color{blue} $1$st antidifference:};
    \node at (3-15/2,-0.7*7+0.02) {\color{blue} $2$nd antidifference:};
    \node at (3-19/2,-0.7*8+0.02) {{\color{green}$(n+3)$th row} $=$ {\color{blue} $3$rd antidifference:}}; 
    
    \node at (1-5/2,-0.7*5) {$17$};
    \node at (2-5/2,-0.7*5) {$17$};
    \node at (3-5/2,-0.7*5) {$15$};
    \node at (4-5/2,-0.7*5) {$12$};
    \node at (5-5/2,-0.7*5) {$9$};
    \node at (1-6/2,-0.7*6) {$-70$};
    \node at (2-6/2,-0.7*6) {$-53$};
    \node at (3-6/2,-0.7*6) {$-36$};
    \node at (4-6/2,-0.7*6) {$-21$};
    \node at (5-6/2,-0.7*6) {$-9$};
    \node at (6-6/2,-0.7*6) {$\color{red}0$};
    \node at (1-7/2,-0.7*7) {$\color{red}0$};
    \node at (2-7/2,-0.7*7) {$-70$};
    \node at (3-7/2,-0.7*7) {$-123$};
    \node at (4-7/2,-0.7*7) {$-159$};
    \node at (5-7/2,-0.7*7) {$-180$};
    \node at (6-7/2,-0.7*7) {$-189$};
    \node at (7-7/2,-0.7*7) {$-189$};
    \node at (1-8/2,-0.7*8) {$\color{red}2017$};
    \node at (2-8/2,-0.7*8) {$2017$};
    \node at (3-8/2,-0.7*8) {$1947$};
    \node at (4-8/2,-0.7*8) {$1824$};
    \node at (5-8/2,-0.7*8) {$1665$};
    \node at (6-8/2,-0.7*8) {$1485$};
    \node at (7-8/2,-0.7*8) {$1296$};
    \node at (8-8/2,-0.7*8) {$1107$};
    
\end{tikzpicture}
    \caption{For $k=3$ and $n=5$, example computation of the $n+3$th row given the $n$th row and the sum of the $(n+2)$th row, here $k=3$ and $n=8$. For each row, the first entry that is assigned to it is in red. The rest of the entries are computed by taking the antidifference of the row above. This gives that the next row of the triangle in Figure~\ref{fig:fmn} is $2017,2017,1947,1824,1665,1485,1296,1107$.}
    \label{fig:comp}
\end{figure}

\subsection{A heuristic calculation of the distribution of $f_k(m,n)$}\label{subsec:heuristic}
In this subsection, we will make a non-rigorous calculation for the distribution of $f_k(m,n)$. For this subsection only, let's assume that $f_k(m,n)$ is asymptotically given by some distribution, in the sense that there is $\varphi_k\colon [0,1]\to \mathbb{R}$ such that $f_k(m,n)$ is close to $\frac{1}{n}\varphi_k\left(\frac{m}{n}\right) f_k(n)$. Note that if $f_k(m,n)$ is to approach some continuous distribution on $[0,1]$, then the normalization $\frac{1}{n}$ is needed in this statement, as the entire mass $f_k(n)$ is divided between the $n$ values $1,\ldots, n$ for $m$. Geometrically, if the values $\frac{f_k(m,n)}{f(n)}$ are supposed to approach areas of columns of width $\frac{1}{n}$ under some density function $\varphi_k$ supported on $[0,1]$, then the area of a column should be $\frac{1}{n}\varphi_k\left(\frac{m}{n}\right)f_k(n)$. 

Under this assumption for the distribution of $f_k(m,n)$, let's see what we can heuristically make of the fact that this distribution should be stable under the finite difference recurrence in Theorem~\ref{thm:mainrec}. Each finite difference can (heuristically) be approximated by the derivative times the gap $\frac{1}{n}$. So heuristically and ignoring terms we anticipate to be lower order, e.g. treating $\frac{1}{n-k}$ as essentially $\frac{1}{n}$, we get \[\frac{d^k \varphi_k(x)}{(dx)^k}\frac{1}{n^k}f_k(n)\approx \varphi_k(x)f_k(n-k).\]

From previous work (see Section~\ref{sec:nasy} and Theorem~\ref{thm:nasy} in particular), it is known that $\frac{f_k(n)}{f_k(n-k)}\approx r_k^k n^k$ for some constant $r_k>0$. Plugging this into our heuristic calculation, we get
\[\frac{d^k \varphi_k(x)}{(dx)^k}\approx \varphi_k(x).\]
Note that 
\[\frac{d^k \varphi_k(x)}{(dx)^k}= r_k^k\varphi_k(x)\]
is a $k$th order differential equation. The boundary conditions of the recurrence given in Subsection~\ref{subsec:fast} suggest the boundary conditions $\varphi'(0)=0, \varphi''(0)=0,\ldots, \varphi^{(k-2)}(0)=0$, and $\varphi^{k-1}(1)=0$, and there is also the normalization $\int_{0}^1\varphi_k(x) dx=1$. We have a $k$th order equation with $k$ boundary conditions, so we would expect there to be a unique solution, although even this would require work to show rigorously.

But for instance for $k=3$, there is indeed a unique solution, and this turns out to be exactly what we will later (in Theorem~\ref{thm:fmnasy}) rigorously find to be the distribution of $f_3(m,n)$. For any $k$, what we will rigorously find to be the distribution of $f_k(m,n)$ satisfies this differential equation with these boundary conditions (so if the solution is unique, then the unique solution of this differential equation is indeed the distribution of $f_k(m,n)$). However, our proof of the distribution theorem (Theorem~\ref{thm:fmnasy}) in Section~\ref{sec:mnasy} is along very different lines. It remains open whether the approach described in this subsection can somehow be made rigorous. Perhaps one could define some operators on a suitable space of functions, and show convergence to a fixed point. However, we have not been able to carry this out.

\subsection{The case of general $I$}\label{subsec:general}

In fact, a similar line of reasoning as in Subsection~\ref{subsec:simprec} works for any $k$-descent set $I$, as long as the part of the permutation where the recurrence is derived is away from $I$. To make this work, we flip our permutations to bring all the $k$-descents to the end, and get the recurrence from the first $k$ indices. We start with the following definition.

\begin{definition}
For $I=\{i_1,\ldots, i_\ell\}\subseteq \mathbb{Z}^+$ and $n\geq \max(I)+k-1$, we define the \emph{$n$-reverse of $I$} to be
\[r_n(I)=\{n+2-k-i_1,\ldots, n+2-k-i_\ell\}.\]
\end{definition}

We have the following simple observation, with the proof omitted. (Reminder: $D_k(w)$ is the set of starting indices of $k$-descents in $w$.)

\begin{observation}
For $w\in S_n$ and letting $\mathrm{rc}(w)\in S_n$ denote the reverse-complement of $w$, i.e. $\left(\mathrm{rc}(w)\right)(i)=n+1-w(n+1-i)$, we have
\[D_k(\mathrm{rc}(w))=r_n(D_k(w)).\]
\end{observation}

The above observation together with the fact that $\mathrm{rc}\colon S_n\to S_n$ is a bijection prove the following remark.
\begin{remark}\label{rmk:bij}
For $k\in \mathbb{Z}^+$, finite $I\subseteq \mathbb{Z}^+$, and $n\geq \max(I)+k-1$,
\[d_k(I,n)=d_k(r_n(I),n).\]
\end{remark}

With this notation, we can state an analog of Theorem~\ref{thm:mainrec} for arbitrary $I$. 

\begin{theorem}\label{thm:otherrec}
For $k\geq 2$, finite $I\subseteq \mathbb{Z}^+$, and $n\geq \max(I\cup \{0\})+k-1$,
\[\Delta^k\left(\left(d_k(r_{n+k}(I),1,n+k),d_k(r_{n+k}(I),2,n+k)\ldots, d_k(r_{n+k}(I),n+k,n+k)\right)\right)\]\[=\left(d_k(r_n(I),1,n),d_k(r_n(I),2,n),\ldots,d_k(r_n(I),n,n)\right).\]
\end{theorem}

\begin{proof}
The proof is essentially identical to the proof of Theorem~\ref{thm:mainrec}.
\end{proof}

\subsection{Fast computation of $d_k(I,n)$}\label{subsec:genfast}
Given Theorem~\ref{thm:otherrec}, one can repeat the argument in Subsection~\ref{subsec:fast} to get a $\Theta_{k,I}(N^2)$ algorithm for computing $d_k(r_n(I),n)$, and hence for computing $d_k(I,n)$ as well, as these are equal by Remark~\ref{rmk:bij}. However, with $t:=\max(I)+k-1$, Theorem~\ref{thm:otherrec} only allows for fast computation once all $d_k(r_n(I),m,n)$ are found for $t\leq n\leq t+k-1$. If these numbers were found naively by checking all permutations, it would take more than $t!$ time, which could be the bottleneck for practical purposes. However, the computation of these initial values can also be done much faster using a dynamic programming approach. Namely, one can start from $n=1$ and go up to $n=t+k-1$ doing the following. For each $n$, $m\leq n$, and $\ell\leq k$, we find the number of permutations in $S_n$ that start with $m$, have an initial decreasing sequence of length $\ell$ (and no initial decreasing sequence of length $\ell+1$), and do not violate the prescribed descent structure so far. These values for $n+1$ can each be found as a sum of values for $n$. For instance, the only way to have an initial decreasing sequence of length $\ell+1$ starting at $m$ is to concatenate a new larger value to the start of a decreasing sequence of length $\ell$ starting from $m'<m$ in the relative ordering on the last $n$ elements.  We will leave the details of figuring out the general case to the interested reader, as it is our opinion that this is easier to understand by giving it some thought, rather than by reading a formal description (for instance, we would need to introduce some new notation just to formally say what it means for a permutation to not violate the $k$-descent structure $I$ for $n<t$).

\subsection{Generating functions for $f_k(m,n)$}\label{subsec:generating}
In this subsection, we discuss generating functions. We are mainly interested in describing the ordinary generating function (o.g.f.) of $f_k(m.n)$.
\begin{definition}
For $k\geq 2$, we let $T_k(x,y)$ be the ordinary generating function for $f_k(m,n)$:
\[T_k(x,y)=\sum_{m,n\geq 1}f_k(m,n)x^m y^n.\]
\end{definition}

To translate Theorem~\ref{thm:mainrec} into the language of generating functions, we will use the following well-known lemma.

\begin{lemma}
Let $k,n\in \mathbb{Z}^+$. Suppose
$\Delta^k\left(\left(a_1,\ldots, a_{n+k}\right)\right)=\left(b_1, \ldots, b_n\right).$
Then for any $m\in [n]$, \[b_m=\sum_{i=0}^k (-1)^{k-i}\binom{k}{i}a_{m+i}.\]
\end{lemma}

\begin{proof}
This is true by induction on $k$. The base case $k=1$ is trivial, and the inductive step is just $\binom{k}{i}+\binom{k}{i-1}=\binom{k+1}{i}$.
\end{proof}

Applying this to the expression in Theorem~\ref{thm:mainrec} and multiplying everything by $x^{m+k} y^{n+k}$, we arrive at the following identity.

\begin{lemma}\label{lem:gen}
For any $k\geq 3$,
\[f_k(m,n)x^{m+k} y^{n+k}=\sum_{i=0}^k (-1)^{k-i}\binom{k}{i}f_k(m+i,n+k)x^{m+k}y^{m+k}.\]
\end{lemma}

Summing this over all pairs $(m,n)\in \mathbb{Z}^+\times\mathbb{Z}^+$ with $1\leq m\leq n$, we get $T_k(x,y)x^k y^k$ on the left-hand side. As for the terms appearing on the right-hand side, these have the form $\sum_{1\leq m\leq n} f_k(m+i,n+k)x^{m+k}y^{n+k}$. This is almost equal to $T_k(x,y)x^{k-i}$, except that all terms with $m\leq i$ or $n\leq k$ or $m\geq n-k+i+1$ are missing. To deal with these missing terms, we proceed with a few more definitions. 


\begin{definition}
For $k\geq 2$, we define $F_k(y)$ to be the ordinary generating function for $f_k(n)$:
\[F_k(y)=\sum_{n\geq 1}f_k(n)y^n.\]
We also define $G_\ell(z)$ to be the ordinary generating function for $f_k(n,n+1-\ell)$:
\[G_{k,\ell}(z)=\sum_{n\geq \ell+1}f_k(n,n+1-\ell)z^{n-1}.\]
\end{definition}

The choices of indices in the definition of $G_{k,\ell}(z)$ may look strange, but this will be a convenient choice for later. To see how all the missing terms can be written in terms of these generating functions, first recall from before that $f_k(1,n)=f_k(2,n)=\cdots=f_k(k-1,n)=f_k(n-1)$ and $f_k(k,n)=f_k(n-1)-f_k(n-k)$. For each fixed value of $m\leq k$, consider the sum of missing terms with this $m$. Because of the fact we just recalled, any such sum can be written as a polynomial in $x,y$ times $F_k(y)$. For the missing terms with $m\geq n-k+1$, one can write these as a polynomial in $x,y$ times $G_{k,n+1-m}(xy)$. As for the missing terms with $n\leq k$, there are only finitely many, so these can be subtracted as a polynomial. Carrying all this out explicitly (e.g. for some particular $k$, such as $k=5$) is a huge mess, as for instance one needs to make sure that terms with $n\leq k$ only get subtracted once (so really these need to be added back in according to how many times each term is overcounted by the $m\leq k$ and $m\geq n-k+1$ sums). Nevertheless, even without computing all the coefficients explicitly, this argument gives us a functional equation for $T_k(x,y)$ of the following form:
\[T_k(x,y)\left(x^k y^k -(1-x)^k\right)=P_k(x,y)F_k(y)+\sum_{\ell=1}^k Q_{k,\ell}(x,y)G_{k,\ell}(xy)+R_k(x,y),\]
where $P_k(x,y)$, $Q_{k,\ell}(x,y)$ (for any $\ell\in [k]$), and $R_k(x,y)$ are polynomials, and we used the fact that the coefficients of $(1-x)^k=\sum_{i=0}^k (-1)^{k-i}\binom{k}{i}x^{k-i}$ match those in Lemma~\ref{lem:gen}. As $x^k y^k -(1-x)^k$ is invertible, we get a functional equation for $T_k(x,y)$.

\begin{proposition}\label{prop:genf}
For any $k\geq 3$, there are polynomials $P_k(x,y)$, $Q_{k,1}(x,y),\ldots, Q_{k,k}(x,y)$, and $R_k(x,y)$ so that
\[T_k(x,y)=\frac{P_k(x,y)F_k(y)+\sum_{\ell=1}^k Q_{k,\ell}(x,y)G_{k,\ell}(xy)+R_k(x,y)}{x^k y^k-(1-x)^k}.\]
\end{proposition}


We will now specialize to the case $k=3$. We let $g_3(n)$ be the number of permutations in $S_n$ with no $3$-descents and no initial descent. Noting that $f_3(n,n)=g_3(n-1)$ (since the $3$-descent-avoiding permutations starting with $n$ are precisely concatenations of $n$ with a permutation on $n-1$ elements that avoids $3$-descents and does not start with a descent), we get that $G_{3,1}$ defined before is also precisely the ordinary generating function for $g_3(n)$, which we will denote $G_3$ from now on for convenience:

\[G_3(z)=G_{3,1}(z)=\sum_{n\geq 1}g_3(n)z^n.\]
This was the motivation for the choice of indexing before. As a side remark, the facts that $f_k(1,n)=f_k(n-1)$ and $f_3(n,n)=g_3(n-1)$ provide another reason to think that the triangle of numbers $f_k(m,n)$ is nice -- namely, for $k=3$, the diagonal of first elements of rows is the sequence $f_3(n)$ (sequence A049774 in OEIS), and the diagonal of last elements of rows is the sequence $g_3(n)$ (sequence A080635 in OEIS), both of which are well-studied.

The following lemma will help us get a more explicit equation for $T_3(x,y)$.

\begin{lemma}
For any $n\geq 3$, $f_3(n,n)=g_3(n-1)$, $f_3(n-1,n)=g_3(n-1)+g_3(n-2)$, and $f_3(n-2,n)=g_3(n-1)+2g_3(n-2)$.
\end{lemma}

\begin{proof}
The first claim was proved earlier. As for $f_3(n-1,n)$, any permutation starting with $n-1$ and having a $3$-descent-avoiding restriction to the last $n-1$ indices that does not start with a descent is counted by $f_3(n-1,n)$, and there are $g_3(n-1)$ such permutations. The only other permutations counted by $f_3(n-1,n)$ start with $n-1$ and have $n$ as the second element, in which case there are $g_3(n-2)$ options for the restriction to the last $n-2$ elements. So $f_3(n-1,n)=g_3(n-1)+g_3(n-2)$.

For $f_3(n-2,n)$, any permutation starting with $n-2$ and having a $3$-descent-avoiding restriction to the last $n-1$ indices that does not start with a descent is counted by $f_3(n-1,n)$, and there are $g_3(n-1)$ such permutations. Any other permutation counted by $f_3(n-1,n)$ has either $n$ or $n-1$ as the second element (and the third element less than the second). The first case is counted by $g_3(n-2)$ as before. For the second case, the third element cannot be $n$, so the third and fourth element must not form a descent (and this is sufficient as well) so this case is also counted by $g_3(n-2)$. Hence, $f_3(n-2,n)=g_3(n-1)+2g_3(n-2)$.
\end{proof}

The above lemma implies that for $k=3$, the $m\geq n-k+1$ missing terms can all be expressed in terms of $G_3$ alone. We worked our way through this calculation, figuring out these polynomials explicitly. The result is the following.

\begin{proposition}\label{prop:gen}
\[T_3(x,y)=\frac{P(x,y)F_3(y)+Q(x,y)G_3(xy)+R(x,y)}{x^3 y^3-(1-x)^3},\]

where $P,Q,R$ are the following polynomials in $x$ and $y$:
\[P(x,y)=xy(x^2y^2-(1-x)^2),\]
\[Q(x,y)=(x-1)x^2y(xy+x-1),\]
\[R(x,y)=(x-1)xy\left((x-1)^2-x^2y^2\right).\]
\end{proposition}

For $k>3$, it remains open whether it is possible to reduce $G_{k,\ell}$ to some small number of generating functions, and whether the coefficient polynomials appearing in Proposition~\ref{prop:genf} can be explicitly understood.

\section{Asymptotics of $f_k(n)$}\label{sec:nasy}

In this section, we discuss the asymptotics of $f_k(n)$ for $k$ fixed and $n\to \infty$. This section is mostly review of work by other authors and well-known methods. In later sections, we will mostly use the following theorem which is a special case of Corollary 1.4. in \cite{kitaev}.

\begin{theorem}[Ehrenborg-Kitaev-Perry \cite{kitaev}]\label{thm:nasy}
For $k\in \mathbb{Z}$, $k\geq 2$, there are $c_k,r_k,\gamma_k\in \mathbb{R}$ with $0<r_k$, $0<c_k$, and $0\leq \gamma_k<1$, such that
\[f_k(n)=n!c_k r_k^n\left(1+O_k(\gamma_k^n)\right).\]
\end{theorem}

\subsection{The value of $r_3$}

In this subsection, we re-derive Theorem~\ref{thm:nasy} for the special case $k=3$, both to find the value of $r_3$, but also as exposition of a nice analytic method for finding asymptotics of generating functions. The reader is referred to \cite{flajolet} for a much more general overview of this method.

We start with the following exponential generating function (e.g.f.) for $f_3(n)$ (OEIS sequence A049774, e.g.f. given by Noam Elkies \cite{oeis}):
\[B(x):=\frac{\sqrt{3}}{2}\frac{e^{x/2}}{\sin\left(\frac{\sqrt{3}}{2}x+\frac{2}{3}\pi\right)}.\]

That is, $f_3(n)$ is $n!$ times the $x^n$ coefficient of $B(x)$. Note that $B(x)$ is a meromorphic function with poles at $x=\frac{2\pi}{3\sqrt{3}}+\ell \frac{2\pi}{\sqrt{3}}$ for all $\ell\in \mathbb{Z}$, and these poles are simple. The two poles with smallest magnitudes are $x_1=\frac{2\pi}{3\sqrt{3}}$ and $x_2=\frac{-4\pi}{3\sqrt{3}}$. For reasons soon to be apparent, we will want to multiply $B(x)$ with a function that cancels out the pole at $x_1$, getting a function which is holomorphic in a disk around $0$ containing $x_1$. We define
\[A_1(x)=\frac{\sqrt{3}}{2}\frac{e^{x/2}}{\sin\left(\frac{\sqrt{3}}{2}x+\frac{2}{3}\pi\right)}\left(\frac{2\pi}{3\sqrt{3}}-x\right).\]

Getting rid of the pole at $x_2$ as well, we further define
\[A_2(x)=\frac{\sqrt{3}}{2}\frac{e^{x/2}}{\sin\left(\frac{\sqrt{3}}{2}x+\frac{2}{3}\pi\right)}\left(\frac{2\pi}{3\sqrt{3}}-x\right)\left(\frac{-4\pi}{3\sqrt{3}}-x\right).\]

We can use this to express $b_n$, the $x^n$ coefficient of $B(x)$, via $A_1(x_1)$. We pick some $R\in \mathbb{R}$ with $|x_2|>R>|x_1|$, so $A_1$ is holomorphic in the disk of radius $R$ around $0$. In this disk, we can write $A_1(x_1)$ as a power series:
\[A_1(x_1)=\sum_{i=0}^\infty a_i x_1^i.\]
On the other hand, we have the formal power series expansion
\[\frac{1}{x_1-x}=\frac{1}{x_1}\left(1+\frac{x}{x_1}+\frac{x^2}{x_1^2}+\cdots\right),\]
from which
\[b_n=\frac{a_n}{x_1}+\frac{a_{n-1}}{x_1^2}+\cdots+\frac{a_0}{x_1^{n+1}}.\]
Hence,
\[b_n x_1^{n+1}=a_0+a_1 x_1+ a_2 x_1^2+\cdots +a_n x_1^n=A_1(x_1)-\sum_{i=n+1}^{\infty}a_i x_1^i,\]
where we used the fact that the power series converges to $A_1(x_1)$ at $x_1$ in the last equality. Since as $n\to \infty$, $\sum_{i=n+1}^\infty a_i x_1^i\to 0$, we get that $b_n x_1^{n+1}\to A_1(x_1)$, from which 
\[b_n\sim \frac{A_1(x_1)}{x_1^{n+1}}.\]
From here, we already get that in Theorem~\ref{thm:nasy}, $r_3=\frac{1}{x_1}=\frac{3\sqrt{3}}{2\pi}$, and $c_3=A_1(x_1)/x_1=\frac{3\sqrt{3}}{2\pi}e^{\frac{\pi}{3\sqrt{3}}}$. By iterating this procedure once more (with $A_1(x)$ in place of $B(x)$ and $A_2(x)$ in place of $A_1(x)$), we get that asymptotically in $i$,
\[a_i\sim \frac{A_2(x_2)}{x_2^{i+1}}.\]
Hence, there is some constant $C$, such that for all $i\geq 0$,
\[|a_i|\leq \frac{C}{|x_2|^{i+1}}.\]

Using this, we can bound the error term in our equation for $b_n$:

\[|b_n x_1^{n+1}-A_1(x_1)|=\left\lvert\sum_{i=n+1}^{\infty}a_i x_1^{i}\right\rvert\leq \sum_{i=n+1}^\infty \frac{C}{|x_2|^{i+1}}x_1^{i}=C' \sum_{i=n+1}^\infty \left(\frac{x_1}{|x_2|}\right)^i=\frac{C'}{2^n}.\]

Hence, we can pick $\gamma_3=\frac{1}{2}$ in Theorem~\ref{thm:nasy}. We will summarize what we just proved in the next proposition.

\begin{proposition}
We define 
\[A_1(x)=\frac{\sqrt{3}}{2}\frac{e^{x/2}}{\sin\left(\frac{\sqrt{3}}{2}x+\frac{2}{3}\pi\right)}\left(\frac{2\pi}{3\sqrt{3}}-x\right)\]
and $x_1=\frac{2\pi}{3\sqrt{3}}$. Then
\[f_3(n)=n!\frac{A_1(x_1)}{x_1^{n+1}}\left(1+O\left(\frac{1}{2^n}\right)\right).\]
\end{proposition}

\subsection{Other $r_k$} We now consider the case of general $k\geq 3$. We start from the following well-known exponential generating function for $f_k(n)$, which appears as Exercise 23.(b) in Chapter $2$ of Stanley's Enumerative Combinatorics 1 \cite{10.5555/2124415}.

\begin{proposition}[\cite{genfunc}]
For any $k\geq 3$, the following is an exponential generating function for $f_k(n)$:
\[B_k(n)=\frac{1}{\sum_{\ell=0}^\infty \frac{x^{k\ell}}{(k\ell)!}-\frac{x^{k\ell+1}}{(k\ell+1)!}}.\]
That is, with $b_{k,n}$ being the $x^n$ coefficient in the power series for $B_k(n)$, we have $f_k(n)=b_{k,n}n!$.
\end{proposition}
Warlimont \cite{warlimont} proves that this exponential generating function has a unique smallest magnitude pole and that this pole is simple. Together with an analogous standard argument to what we just showed for $k=3$, this implies that $r_k$ (in Theorem~\ref{thm:nasy}) is the reciprocal of the smallest magnitude root of $\sum_{\ell=0}^\infty \frac{x^{k\ell}}{(k\ell)!}-\frac{x^{k\ell+1}}{(k\ell+1)!}$. Warlimont also provides bounds on $r_k$. We state all of this in the next proposition.

\begin{proposition}[Warlimont \cite{warlimont}]\label{prop:war}
For any $k\geq 4$, the constant $r_k$ in Theorem~\ref{thm:nasy} is the unique smallest magnitude pole of $B_k$. Furthermore, we have the following bounds:
\[1+\frac{1}{k!}\left(1-g(k)\right)\leq \frac{1}{r_k}\leq 1+\frac{1}{k!}\left(1+h(k)\right),\]
where
\[g(k)=\frac{k!+1}{(k+1)!+1},\hspace{5mm}h(k)=\frac{2(k+1)}{k!-2(k+1)}.\]
\end{proposition}

The first three values are $\frac{1}{r_3}=\frac{2\pi}{3\sqrt{3}}=1.209199576\ldots$, $\frac{1}{r_4}=1.038415637\ldots$, and $\frac{1}{r_5}=1.007187547786\ldots$ (given by Kotesovec on OEIS sequences A049774, A117158, and A177523, respectively) \cite{oeis}.

We finish this section with another remark. There is a trick which lets us rewrite the aforementioned infinite sum in a finite form. We start from an identity which can be proven by expanding all terms on the RHS as infinite series:
\[\sum_{\ell=0}^\infty \frac{x^{k\ell}}{(k\ell)!}=\frac{1}{k}\left(e^x+e^{\omega_k x}+\cdots+e^{\omega_k^{k-1}}x\right),\]
where $\omega_k=e^{\frac{2\pi i}{k}}$. Integrating both sides, we get
\[\sum_{\ell=0}^\infty \frac{x^{k\ell+1}}{(k\ell+1)!}=\frac{1}{k}\left(e^x+\frac{1}{\omega_k}e^{\omega_k x}+\cdots+\frac{1}{\omega_k^{k-1}} e^{\omega_k^{k-1}x}\right).\]

Subtracting the second from the first, we get
\[\frac{1}{B_k(x)}=\frac{1}{k}\left(\left(1-\frac{1}{\omega_k}\right)e^{\omega_k x}+\left(1-\frac{1}{\omega_k^2}\right)e^{\omega_k^2 x}+\cdots+\left(1-\frac{1}{\omega_k^{k-1}}\right)e^{\omega_k^{k-1}x}\right).\]

\section{Asymptotics of $f_k(m,n)$}\label{sec:mnasy}

\subsection{Two propositions on $f_k(m,n)$}
In this subsection, we present two propositions on $f_k(m,n)$. The first one is a nice fact about $f_k(m,n)$ which will be needed to finish the proof of Theorem~\ref{thm:fmnasy} later.

\begin{proposition}\label{prop:decr}
For any $k\geq 2$ and $n\in \mathbb{Z}^+$,
\[f_k(1,n)\geq f_k(2,n)\geq \cdots \geq f_k(n,n).\]
\end{proposition}

\begin{proof}
For $2\leq m\leq n$ and any $w\in \mathcal{D}_k(\emptyset,m,n)$, switching $m$ and $m-1$ in $w$ gives $w'\in \mathcal{D}_k(\emptyset,m-1,n)$, since no $k$-descent can be created by this operation. Furthermore, $w\mapsto w'$ is injective because $w$ can be uniquely recovered from $w'$ by switching $m$ and $m-1$ in $w'$. Hence, $f_k(m,n)\leq f_k(m-1,n)$.
\end{proof}

The second proposition will be crucial in deriving the asymptotic distribution of $f_k(m,n)$.

\begin{proposition}\label{prop:fmn}
For any $k\geq 3$ and $m\leq n\in \mathbb{Z}^+$,
\[f_k(m,n)=\binom{m-1}{0}f_k(n-1)-\binom{m-1}{k-1}f_k(n-k)+\binom{m-1}{k}f_k(n-k-1)-\binom{m-1}{2k-1}f_k(n-2k)+\cdots.\]
\end{proposition}

For $k=3$, one can prove this by observing that in Proposition~\ref{prop:gen}, the $G(xy)$ and $R(x,y)$ terms will only contribute to coefficients of $x^m y^n$ with $m>n$. So $f_3(m,n)$ is just given by the coefficients of $\frac{Q(x,y)F(y)}{P(x,y)}$. For general $k$, we give the following combinatorial proof.

\begin{proof}
Let us consider the right-hand side of the equation we want to prove. We think of the first term, $\binom{m-1}{0}f_k(n-1)$, as counting all permutations that start with $m$ and for which the restriction to the other $n-1$ indices is a permutation in $\mathcal{D}_k(\emptyset,n-1)$. We think of the second term, $\binom{m-1}{k-1}f_k(n-k)$, as counting all permutations that start with a decreasing sequence of $k$ elements beginning at $m$, and the restriction to the other $n-k$ indices is a permutation in $\mathcal{D}_k(\emptyset,n-k)$. We think of the next term, $\binom{m-1}{k}f_k(n-k-1)$ as the same except the initial sequence is now of length $k+1$, and so on. Let us consider how many times each permutation in $S_n$ gets counted, taking signs into account. If a permutation does not start with $m$, then it is clearly not counted by any term. If a permutation starts with $m$ and contains a $k$-descent somewhere not in an initial decreasing sequence, then it does not get counted by any term. If a permutation starts with $m$, avoids $k\ldots 1$ except in the initial decreasing sequence, and has an initial decreasing sequence of length $t\geq k$, it gets counted in exactly the terms with $\binom{m-1}{\ell}$ with the two maximal values $\ell<t$; since these have opposing signs, these counts cancel each other. If a permutation starts with $m$ and avoids $k$-descents, then it is counted exactly once, namely by just the first term. This covers all the cases. Hence, the right-hand side counts the number of permutations starting with $m$ and avoiding $k$-descents, which is equal to $f_k(m,n)$ by definition. So the right-hand side is equal to the left-hand side, completing the proof.
\end{proof}

We remark that one can also derive the $k$th difference equation (Theorem~\ref{thm:mainrec}) from Proposition~\ref{prop:fmn}.

\subsection{The asymptotic distribution of $f_k(m,n)$}
We now prove the following theorem which describes the asymptotic distribution of $f_k(m,n)$. The content of this theorem is that as $n\to \infty$, the mass of $f_k(n)$ is distributed among $f_k(1,n),\ldots, f_k(n,n)$ according to an explicit distribution $\varphi_k\left(\frac{m}{n}\right)$ -- crucially, this distribution does not depend on $n$ (after the appropriate normalization of $\frac{1}{n}$).

\begin{theorem}\label{thm:fmnasy}
For all $k\geq 3$, with $r_k$ from Theorem~\ref{thm:nasy}, for all $m,n\in \mathbb{Z}^+$ with $1\leq m\leq n$,
\[\frac{nf_k(m,n)}{f_k(n)}=\varphi_k\left(\frac{m}{n}\right)\left(1+O_k\left(n^{-0.49}\right)\right),\]
where 
\[\varphi_k\left(x\right)=\frac{1}{r_k}\left(1-\frac{(x/r_k)^{k-1}}{(k-1)!}+\frac{(x/r_k)^k}{k!}-\frac{(x/r_k)^{2k-1}}{(2k-1)!}+\frac{(x/r_k)^{2k}}{(2k)!}-\cdots \right).\]
\end{theorem}

We note that the exponent $-0.49$ is just chosen for clarity, and our proof really gives something slightly stronger.
\begin{remark}
In Theorem~\ref{thm:fmnasy}, the exponent $-0.49$ can be replaced by any $\alpha>-0.5$.
\end{remark}

Given Proposition~\ref{prop:fmn} and Theorem~\ref{thm:nasy}, the proof of Theorem~\ref{thm:fmnasy} is just algebra and analysis. 

\begin{proof}[Proof of Theorem~\ref{thm:fmnasy}]
Plugging the expression for $f_k(n)$ from Theorem~\ref{thm:nasy} into the expression for $f_k(m,n)$ in Proposition~\ref{prop:fmn}, we get
\[\frac{n}{f_k(n)}f_k(m,n)=\binom{m-1}{0}r_k\left(1+O(\gamma_k^{n-1})\right) -\binom{m-1}{k-1}\frac{(n-k)!}{(n-1)!}r_k^{-k}\left(1+O(\gamma_k^{n-k})\right)\]\[+\binom{m-1}{k}\frac{(n-k-1)!}{(n-1)!}r_k^{-k-1} \left(1+O(\gamma_k^{n-k-1})\right)-\binom{m-1}{2k-1}\frac{(n-2k)!}{(n-1)!}r_k^{-2k}\left(1+O(\gamma_k^{n-2k})\right)+\cdots.\]

Our proof strategy will be to first show that the terms with $\ell\geq \log{n}$ are negligible, then estimate the terms with $\ell\leq \log{n}$ just for the case $m\geq \sqrt{n}$, and then complete the proof for the remaning $m<\sqrt{n}$ case using analytic arguments and Proposition~\ref{prop:decr}. We will state whenever we restrict to a particular case. We start by consider a general term; it has the following form:
\[\binom{m-1}{\ell}\frac{(n-\ell-1)!}{(n-1)!}r_k^{-\ell-1}\left(1+O(\gamma_k^{n-\ell})\right)\]
\[=\frac{(m-1)\cdots(m-\ell)}{(n-1)\cdots(n-\ell)}\frac{r_k^{-\ell-1}}{\ell!}\left(1+O(\gamma_k^{n-\ell})\right).\]

We use the fact that $\ell!\geq \left(\frac{\ell}{e}\right)^n$, the fact that $m\leq n$ implies that $\frac{m-i}{n-i}\leq 1$, and the fact that $\ell\leq n$ implies that $1+O(\gamma_k^{n-\ell})\leq C$ (where the constant is independent of $\ell$) to upper-bound such a term:
\[\frac{(m-1)\cdots(m-\ell)}{(n-1)\cdots(n-\ell)}\frac{r_k^{-\ell-1}}{\ell!}\left(1+O(\gamma_k^{n-\ell})\right)\leq \frac{C}{r_k}\left(\frac{e/r_k}{\ell}\right)^{\ell}.\]
For $\ell\geq \log n$, we further bound this:
\[\frac{C}{r_k}\left(\frac{e/r_k}{\ell}\right)^{\ell}\leq\frac{C}{r_k}\left(\frac{e/r_k}{\log n}\right)^{\log n}=\frac{C}{r_k}n^{\log \frac{e/r_k}{\log n}}=\frac{C}{r_k}n^{-\log \frac{\log n}{r_k/e}}.\]
Hence, the contribution of all terms with $\ell\geq \log n$ is at most $n$ times the contribution of one such term (since there are at most $n$ terms), totaling to
\[O\left(n^{-\log \frac{\log n}{r_k/e}+1}\right).\]
Let us now focus on the case $m\geq \sqrt{n}$ and $\ell<\log n$. We bring our attention back to a general term. Note that
\[\binom{m-1}{\ell}=\frac{m^\ell}{\ell!}\left(1+O\left(\frac{\ell^2}{m}\right)\right)\]
and
\[\frac{(n-\ell-1)!}{(n-1)!}=n^{-\ell}\left(1+O\left(\frac{\ell^2}{n}\right)\right).\]

Hence, with these bounds on $\ell$ and $m$, a general term is
\[\binom{m-1}{\ell}\frac{(n-\ell-1)!}{(n-1)!}r_k^{-\ell-1}\left(1+O(\gamma_k^{n-\ell})\right)=\frac{1}{r_k}\frac{\left(\frac{m}{n}\frac{1}{r_k}\right)^\ell}{\ell!}\left(1+O\left(n^{-0.49}\right)\right).\]
Now, comparing the sum of the first $\log n$ terms in the initial sum for $f_k(m,n)$ with the sum
\[\frac{1}{r_k}\left(1-\frac{\left(\frac{m}{n}\frac{1}{r_k}\right)^{k-1}}{(k-1)!}+\frac{\left(\frac{m}{n}\frac{1}{r_k}\right)^k}{k!}-\frac{\left(\frac{m}{n}\frac{1}{r_k}\right)^{2k-1}}{(2k-1)!}+\cdots\right),\]
where the sum goes up to the largest $\ell<\log n$, we note that the difference is upper bounded by $\left(1+O\left(n^{-0.49}\right)\right)$ times the sum of absolute values of these terms. It is still upper bounded by the same thing with the finite sum replaced by an infinite sum, which is equal to some constant between $\frac{1}{r_k}$ and $\frac{1}{r_k}e^{\frac{1}{r_k}}$ (just by comparing terms). This observation together with our previous bound on the contribution of terms with $\ell\geq \log n$ implies that
\[\frac{nf_k(m,n)}{f_k(n)}=\frac{1}{r_k}\left(1-\frac{\left(\frac{m}{n}\frac{1}{r_k}\right)^{k-1}}{(k-1)!}+\frac{\left(\frac{m}{n}\frac{1}{r_k}\right)^k}{k!}-\frac{\left(\frac{m}{n}\frac{1}{r_k}\right)^{2k-1}}{(2k-1)!}+\cdots\right)+O(n^{-0.49}),\]
where we used the fact that $O\left(n^{-\log \frac{\log n}{r_k/e}+1}\right)=O(n^{-0.49})$. Now we proceed to bound the difference between the series cut off at $\log n$ and the corresponding infinite series. The sum of absolute values of tail terms (after $\log n$) can be upper-bounded by a geometric series with first term $\frac{1}{r_k}\frac{\left(\frac{m}{n}\frac{1}{r_k}\right)^{\log n}}{(\log n)!}$ and ratio $\frac{\frac{m}{n}\frac{1}{r_k}}{\log n}$. This is $O\left(\left(\frac{e/r_k}{\log n}\right)^n\right)=O(n^{-0.49})$, as argued before. Hence,
\[\frac{nf_k(m,n)}{f_k(n)}=\varphi_k\left(\frac{m}{n}\right)+O(n^{-0.49}),\]
with
\[\varphi_k(x)=\frac{1}{r_k}\left(1-\frac{\left(x/r_k\right)^{k-1}}{(k-1)!}+\frac{\left(x/r_k\right)^k}{k!}-\frac{\left(x/r_k\right)^{2k-1}}{(2k-1)!}+\cdots\right)\]
with the series being infinite now.

We still need to deal with the case $m<\sqrt{n}$. In order to do so, we analyze the function $\varphi_k(x)$. Note that for $x\in [0,1]$, $\varphi_k(x)$ is non-increasing. One can see this by showing that the derivative is non-positive by taking the derivative of the series, pairing up consecutive terms, and using the fact that $\frac{1}{r_k}<\sqrt{2}$, implied by Proposition~\ref{prop:war} and a manual computation for $k=3$, to show that each successive term is smaller in magnitude. A similar pairing argument gives that the second derivative is negative, and that $\varphi_k(1)>0$. We now deal with the case $m\leq \sqrt{n}$. In that case, by Proposition~\ref{prop:decr}, the value of $n\frac{f_k(m,n)}{f_k(n)}$, is at most $\frac{nf_k(1,n)}{f_k(n)}=\frac{nf_k(n-1)}{f_k(n)}=\frac{1}{r_k}\left(1+O(\gamma_k^{n-1})\right)$, and at least $\frac{n f_k(\lceil\sqrt{n}\rceil,n)}{f_k(n)}$. By what we have already proved, this lower bound is $\varphi_k\left(\frac{\lceil\sqrt{n}\rceil}{n}\right)+O(n^{-0.49})$. Since the first and second derivatives of $\varphi_k(x)$ are both negative, the magnitude of the first derivative of $\varphi_k$ is upper-bounded by $|\varphi_k'(1)|$. This is a constant, so $\varphi_k\left(\frac{\lceil\sqrt{n}\rceil}{n}\right)+O(n^{-0.49})=\varphi_k(0)+O(n^{-0.5})=\frac{1}{r_k}+O(n^{-0.5})$. Also, for $m\leq \sqrt{n}$, $\varphi_k(\frac{m}{n})=\frac{1}{r_k}+O(n^{-0.5})$. By combining the upper and lower bounds with this estimate, we thus get $n\frac{f_k(m,n)}{f_k(n)}=\varphi_k\left(\frac{m}{n}\right)+O(n^{-0.49})$.

So far, we have proved that for all $m\leq n$,
\[\frac{nf_k(m,n)}{f_k(n)}=\varphi_k\left(\frac{m}{n}\right)+O(n^{-0.49}).\]
We can finish the proof by noting that there is a uniform lower bound on $\varphi_k(x)$, namely $\varphi_k(x)>\varphi_k(1)>0$. With this, we arrive at the desired result:
\[\frac{nf_k(m,n)}{f_k(n)}=\varphi_k\left(\frac{m}{n}\right)\left(1+O(n^{-0.49})\right).\]

\end{proof}

Before moving on, we note that in analogy to what we did at the end of Section~\ref{sec:nasy}, one can instead write $\varphi_k(x)$ in the following finite form.

\begin{proposition}
We let $\omega_k=e^{\frac{2\pi i}{k}}$. For any $k\geq 3$,
\[\varphi_k(x)=\frac{1}{k r_k}\left((1-\omega_k)e^{\omega_kx}+\left(1-\omega_k^2\right)e^{\omega_k^2 x}+\cdots+\left(1-\omega_k^{k-1}\right)e^{\omega_k^{k-1}x}\right).\]
\end{proposition}

By comparing power series expansions, one can write down something quite simple for the case $k=3$.

\begin{proposition}
\[\varphi_3(x)=\frac{4\pi}{9}e^{-\pi x/ (3\sqrt{3})}\sin((x+1)\pi/3)\]
\end{proposition}
Figure~\ref{fig:plots} shows a plot of $\varphi_3$. 

We will now argue that as $k\to \infty$, the sequence of functions $\varphi_k$ converges pointwise (and uniformly) to the constant function $1$. One can prove this by using Warlimont's bounds on $r_k$ from Proposition~\ref{prop:war} to bound the contribution of all terms after the first term of the series expansion for $\varphi_k$ given in Theorem~\ref{thm:fmnasy}. Namely, all terms appearing in the series expansion for $\varphi_k(x)$ also appear in the series expansion for $e^{x/r_k}$, and as $\frac{1}{r_k}<\sqrt{2}$, the absolute value of each term is upper bounded by the corresponding term in the series expansion for $e^{\sqrt{2}}$. Since this series expansion converges, the tail sum goes to $0$, and as $k\to \infty$, all non-zero terms of the expansion for $\varphi_k$ are contained in a tail sum starting further along, and hence have contribution going to $0$ (and this is uniform over $x\in [0,1]$; alternatively, one see that pointwise convergence implies uniform convergence from the fact that $\varphi_k(0)=\frac{1}{r_k}$ and $\varphi_k(x)$ is decreasing in $[0,1]$, so it suffices to show pointwise convergence to $\frac{1}{r_k}$ for $x=1$).

So $\lim_{k\to \infty}\left(\varphi_k-\frac{1}{r_k}\right)$ is $0$. Again, using Warlimont's bounds (or just that $\varphi_k$ has integral $1$), $\frac{1}{r_k}\to 1$, so we get that $\lim_{k\to\infty}\varphi_k$ is $1$. One would expect that avoiding $k$-descents says less and less about the first element of a permutation as $k$ becomes larger, so this result matches with intuition.

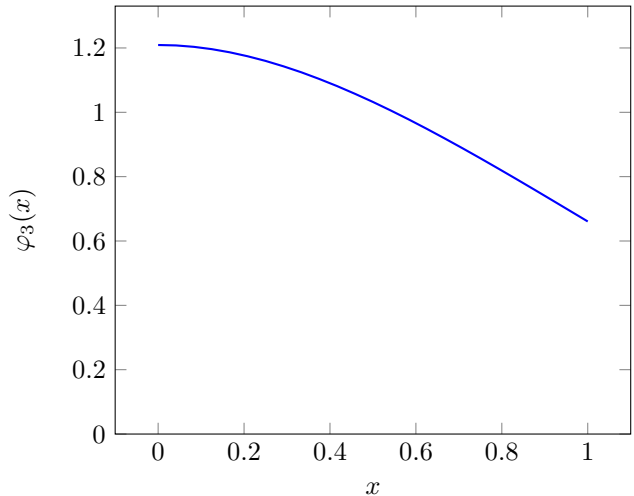
\begin{figure}
    \centering
    \begin{tikzpicture}
  \begin{axis}[domain=0:1,
    xlabel=$x$,
    ylabel={$\varphi_3(x)$},
    ymin=0,
    every axis plot/.append style={thick}
    ] 
    \addplot+[mark=none] {4/9*pi*e^(-pi*x/3^(3/2))*sin(deg(pi*(x+1)/3))}; 
  \end{axis}
\end{tikzpicture}

    \caption{Plot of $\varphi_3$}
    \label{fig:plots}
\end{figure}

\section{Asymptotics of $d_k(I,n)$}\label{sec:dasy}
In this section, we prove that if we fix a finite set $I\subseteq\mathbb{Z}^+$, then $d(I,n)$ is asymptotically given by an integral formula, and that this allows for efficient determination of the asymptotics of $d(I,n)$. As a consequence, we will prove the following theorem.

\begin{theorem}\label{thm:dasy}
For any $k\geq 3$ and finite $I\subseteq \mathbb{Z}^+$, there is a constant $c_{I,k}$ such that
\[d_k(I,n)=c_{I,k} f_k(n)\left(1+O(n^{-0.49})\right).\]
\end{theorem}
Corollaries of this theorem include three conjectures by Zhu; these are Conjecture 3.2, Conjecture 6.4, and Conjecture 6.5 (which is a generalization of Conjecture 6.4) \cite{zhu2019enumerating}. Again, the exponent $-0.49$ can be replaced with anything strictly greater than $-0.5$, as will be evident from our proof (together with the fact that an analogous statement is true for Theorem~\ref{thm:fmnasy}).

Throughout this section, we will think of $I$ as being fixed, and we let $t=\max(I)+k-1$ (or in other words, $t$ is the index of the end of the last $k$-descent). We now start with discussion that will lead to a proof of Theorem~\ref{thm:dasy}.
Begin by noting that for any permutation $w\in\mathcal{D}_k(I,n)$, the restriction of $w$ to the first $t$ indices has descent set $I$: it is an element of $\mathcal{D}_k(I,t)$. In particular, note that it ends with a $k$-descent (this will be useful soon). The restriction of $w$ to the last $n-t$ indices has no $k$-descents: it is an element of $\mathcal{D}_k(\emptyset,n-t)$.

On the other hand, we can construct a (unique) permutation in $S_n$ by picking some $\tau\in \mathcal{D}_k(I,t)$ to be its restriction to the first $t$ indices; picking some $v\in \mathcal{D}_k(\emptyset,n-t)$ to be its restriction to the last $n-t$ indices; and picking a set of images for the first $t$ indices -- we are choosing a $t$-element subset of $[n]$, so there are $\binom{n}{t}$ choices for this last step. All in all, one can construct $d_k(I,t)\cdot f_k(n-t)\cdot \binom{n}{k}$ distinct permutations this way, and by what we argued before, all permutations in $d_k(I,n)$ are among these.



Now, the first piece of bad news is that not every $\tau\in \mathcal{D}_k(I,n)$, $v\in \mathcal{D}_k(\emptyset,n-t)$, and choice of $t$ elements from $[n]$ give a permutation $w\in \mathcal{D}_k(I,n)$. The problem is that although all indices in $I$ are starts of $k$-descents in $w$, it is possible that $w$ has additional $k$-descents starting at some other indices; namely, $w$ could have a $k$-descent starting at $\max(I)+1$ or $\max(I)+2$ or $\ldots$ or $\max(I)+k-1=t$. The first piece of good news is that such an unprescribed $k$-descent occurs in a $w$ constructed this way if and only if $w(t)>w(t+1)$ (this is a consequence of $\tau$ ending with a $k$-descent). In fact, $w(t)>w(t+1)$ will turn out to be the sort of event whose probability we can find using our knowledge of the asymptotic distribution of $f_k(m,n)$, i.e. Theorem~\ref{thm:fmnasy}. 

The second piece of bad news is that the permutations in $\mathcal{D}_k(I,t)$ might in general be hard to describe. To counteract this, we have our second piece of good news: as $I$ is fixed, there are only finitely many $\tau\in \mathcal{D}_k(I,t)$ (namely, no more than $t!$). We will first count the number of permutations $w\in \mathcal{D}_k(I,n)$ that start with a fixed $\tau\in \mathcal{D}_k(I,t)$, and sum over all $\tau\in \mathcal{D}_k(I,t)$ later. From now on, we will think of $\tau$ as being fixed, and we will let $s=s(\tau)$ be the last element of $\tau$, i.e. $s:=\tau(t)$. Now, we construct a random permutation $w\in S_n$ by picking a uniformly random set of $t$ elements from $[n]$ for the (unordered) set of values of $w(1),\ldots, w(t)$, setting their relative order in $w$ to be $\tau$, and then picking a random $v\in \mathcal{D}_k(\emptyset,n-t)$ to be the restriction of $w$ to the last $n-t$ indices. Repeating what we observed before, we get a permutation $w\in \mathcal{D}_k(I,n)$ iff $w(t)<w(t+1)$. So the number of permutations we get this way is

\[\binom{n}{t}f_k(n-t)\mathbb{P}\left(w(t)<w(t+1)\right).\]

Our next goal is to understand $\mathbb{P}(w(t)<w(t+1))$. First off, we essentially know the distribution of $w(t+1)$, as we know the distribution of the first element (Theorem~\ref{thm:fmnasy}) of the restriction to the last $n-t$ indices, and the potential shift by at most $t$ -- depending on the choice of the $t$-element subset of $[n]$ -- is asymptotically small (namely, $O(1/n)$). We will now show that the distribution of $w(t)$ is also simple. The idea is that as $n\to \infty$ but $t$ stays constant, picking $t$ elements of $[n]$ is essentially equivalent to picking $t$ uniform $[0,1]$ random variables (and multiplying each by $n$, and rounding appropriately). We really only care about the value of the $s$th largest of these. In the case of $t$ uniform $[0,1]$ random variables, the $s$th largest has the following well-known distribution.

\begin{proposition}
Let $U_1,U_2,\ldots,U_t$ be independent uniform $[0,1]$ random variables. The $s$th largest of these is a random variable which we denote $U^t_{(s)}$ and call the \emph{$s$th order statistic}. At $y\in [0,1]$, the probability density function of $U^t_{(s)}$ is 
\[\Phi^t_s(u)=\frac{t!}{(s-1)!(t-s)!}y^{s-1}(1-y)^{t-s}.\]
\end{proposition}

The next lemma says that in the discrete case, i.e. picking a $t$-element subset of $[n]$, the $s$th largest has distribution close to $\Phi^t_s(u)$ (rescaled appropriately).

\begin{lemma}\label{lem:os}
Let $t\geq 3$ and $1\leq s\leq t$ be fixed, and $n\in \mathbb{Z}^+$. Let $\mathcal{Y}$ be a uniform random $t$-element subset of $[n]$, and let $Y_{(s)}^t$ be the $s$th largest element of $\mathcal{Y}$. Then
\[\mathbb{P}\left(Y_{(s)}^t=\ell\right)=\frac{1}{n}\Phi_s^t\left(\frac{\ell}{n}\right)+O(n^{-1.5}).\]
\end{lemma}
\begin{proof} Simply by expanding the brackets and replacing in $\Phi_s^t$, we begin by noting that it suffices to prove the following.
\[\mathbb{P}\left(Y_{(s)}^t=\ell\right)=\frac{1}{n}\frac{t!}{(s-1)!(t-s)!}\left(\left(\frac{\ell}{n}\right)^{s-1}\left(1-\frac{\ell}{n}\right)^{t-s}+O(n^{-0.5})\right).\]
Let's do some counting. If the $s$th largest element is $\ell$, then there are a total of $\binom{\ell-1}{s-1}$ choices for the bottom $s-1$ elements and $\binom{n-\ell}{t-s}$ choices for the top $t-s$ elements. Hence,
\[\mathbb{P}\left(Y_{(s)}^t=\ell\right)=\frac{\binom{\ell-1}{s-1}\binom{n-\ell}{t-s}}{\binom{n}{t}}\]\[=\frac{1}{n}\frac{t!}{(s-1)!(t-s)!}\frac{(\ell-1)(\ell-2)\cdots(\ell-s+1)(n-\ell)(n-\ell-1)\cdots(n-\ell-t+s+1)}{(n-1)(n-2)\cdots(n-t+1)}.\]

We note the identical prefactor in the lemma and the above expression, and we proceed to compare the last terms of the two products. Namely, it remains to show that
\[\frac{(\ell-1)(\ell-2)\cdots(\ell-s+1)(n-\ell)(n-\ell-1)\cdots(n-\ell-t+s+1)}{(n-1)(n-2)\cdots(n-t+1)}\]\[=\left(\frac{\ell}{n}\right)^{s-1}\left(1-\frac{\ell}{n}\right)^{t-s}+O(n^{-0.5}).\]
If $\ell\leq \sqrt{n}$ and $s\geq 2$, we are automatically done since $\frac{\ell-1}{n-1}$ on the left-hand side and $\frac{\ell}{n}$ on the right-hand side already imply that both terms are $O(n^{-0.5})$. If $\ell\leq \sqrt{n}$ and $s=1$, then $\frac{n-\ell-i}{n-j}=\left(1-\frac{\ell}{n}\right)\left(1+O(n^{-1})\right)$, and the product of finitely many such terms still gives a multiplicative error term of $(1+O(n^{-1}))$, which is better than the desired bound. By symmetry, the case $\ell\geq n-\sqrt{n}$ is also covered now.

As for the case $\sqrt{n}< \ell<n-\sqrt{n}$, we then have $\frac{\ell-i}{n-j}=\frac{\ell}{n}(1+O(n^{-0.5}))$, and similarly for $\frac{n-\ell-i}{n-j}$. The finitely many $(1+O(n^{-0.5}))$ terms still multiply to a $(1+O(n^{-0.5}))$ term. This finishes the case check and the proof of the lemma.
\end{proof}

We note that with the notation from before, i.e. $\tau$ being the fixed relative order for the first $t$ indices, and $v$ being the $k$-descent-avoiding restriction to the last $n-t$ indices, $v(1)\leq w(t)\leq v(1)+t$, so
\[w(t)\leq v(1)\implies w(t)<w(t+1)\implies w(t)\leq v(1)+t,\]
from which
\[\mathbb{P}(w(t)\leq v(1))\leq \mathbb{P}(w(t)<w(t+1))\leq \mathbb{P}(w(t)\leq v(1)+t).\]
This is useful as it lets us deal with the (otherwise) inconvenient detail that we can understand $w(t)$ and $v(1)$, but we wish to compare $w(t)$ and $w(t+1)$, and $w(t+1)\neq v(1)$ (instead, $w(t+1)$ can be anything in the range $v(1),v(1)+1,\ldots, v(1)+t)$. We now finally get to use our machinery for $w(t)$ and $v(1)$, i.e. Lemma~\ref{lem:os} and Theorem~\ref{thm:fmnasy}. We start from
\[\mathbb{P}\left(w(t)\leq v(1)\right)=\sum_{m=1}^{n-t}\frac{f_k(m,n-t)}{f_k(n-t)}\sum_{\ell=1}^{m}\mathbb{P}\left(Y^t_{(s)}=\ell\right).\]
Applying Lemma~\ref{lem:os}, we get
\[\mathbb{P}\left(w(t)\leq v(1)\right)=\sum_{m=1}^{n-t}\frac{f_k(m,n-t)}{f_k(n-t)}\sum_{\ell=1}^{m}\frac{1}{n}\Phi_s^t\left(\frac{\ell}{n}\right)+O(n^{-1.5}).\]
We sum the $O(n^{-1.5})$ terms up into a $O(n^{-0.5})$ and note that the remaining inner sum is a Riemann sum for the integral $\int_0^{\frac{m}{n}}\Phi_s^t(y)$. Since $\Phi_s^t(y)$ is a polynomial, it is continuously differentiable on $[0,1]$, and by compactness of $[0,1]$ its derivative is bounded, so the difference between our integral and our Riemann sum is $n\cdot O(1/n^2)=O(1/n)$. Hence, we get
\[\mathbb{P}\left(w(t)\leq v(1)\right)=O(n^{-0.5})+\sum_{m=1}^{n-t}\frac{f_k(m,n-t)}{f_k(n-t)}\int_0^{\frac{m}{n}}\Phi_s^t(y).\]

We now use Theorem~\ref{thm:fmnasy}, getting
\[\mathbb{P}\left(w(t)\leq v(1)\right)=O(n^{-0.5})+\sum_{m=1}^{n-t}\frac{1}{n-t}\varphi_k\left(\frac{m}{n-t}\right)\left(1+O\left(n^{-0.49}\right)\right)\int_0^{\frac{m}{n}}\Phi_s^t(y).\]
We again pull out a total additive error term of $(n-t)\cdot\frac{1}{n-t}\cdot O\left(n^{-0.49}\right)=O\left(n^{-0.49}\right)$, which we can do since compactness of $[0,1]$ implies that $\varphi_k(\frac{m}{n})\int_0^{\frac{m}{n}}\Phi_s^t(y)$ is bounded. We then again have a Riemann sum, this time for the outer integral in $\int_0^1\varphi_k(x)\int_{0}^{x\frac{n-t}{n}} \Phi_s^t(y)$. Since $\Phi_s^t$ is continuous and bounded (by compactness of $[0,1]$) and $\varphi_k$ is continuously differentiable and hence has bounded derivative, the Riemann summed function $\varphi_k(x)\int_{0}^{x\frac{n-t}{n}} \Phi_s^t(y)$ is differentiable and has bounded derivative on $[0,1]$. So the difference between our Riemann sum and our integral is $O(1/n)$ as before. Hence,

\[\mathbb{P}\left(w(t)\leq v(1)\right)=O(n^{-0.49})+\int_0^1\varphi_k(x)\int_{0}^{x\frac{n-t}{n}} \Phi_s^t(y).\]

Finally, by boundedness of $\varphi_k(x)\Phi_s^t(y)$ in the compact $[0,1]^2$, the integral over a measure $O(1/n)$ subset of $[0,1]^2$ is itself $O(1/n)$, so
\[\mathbb{P}\left(w(t)\leq v(1)\right)=O(n^{-0.49})+\int_0^1\varphi_k(x)\int_{0}^{x}\Phi_s^t(y).\]

By an analogous argument, we can also get
\[\mathbb{P}(w(t)\leq v(1)+t)=O(n^{-0.49})+\int_0^1\varphi_k(x)\int_{0}^{x}\Phi_s^t(y).\]
Combining these two, we get
\[O(n^{-0.49})+\int_0^1\varphi_k(x)\int_{0}^{x}\Phi_s^t(y)\leq \mathbb{P}(w(t)<w(t+1))\leq O(n^{-0.49})+\int_0^1\varphi_k(x)\int_{0}^{x}\Phi_s^t(y).\]
Hence,
\[\mathbb{P}(w(t)<w(t+1))=O(n^{-0.49})+\int_0^1\varphi_k(x)\int_{0}^{x}\Phi_s^t(y).\]
This integral is a positive constant (it is nonzero since $\varphi_k(x)$ is bounded below and $\Phi^t_s(y)$ integrates to $1$), so we can make the error term multiplicative:
\[\mathbb{P}(w(t)<w(t+1))=\left(1+O(n^{-0.49})\right)+\int_0^1\varphi_k(x)\int_{0}^{x}\Phi_s^t(y).\]

Finally coming back to what we promised a long time ago, we sum over all $\tau\in \mathcal{D}_k(I,t)$, and arrive at the following theorem, which is an integral formula for $d_k(I,n)$.

\begin{theorem}\label{thm:precdasy}
For fixed $k\geq 3$, fixed finite $I\subseteq \mathbb{Z}^+$, and asymptotically in $n\in \mathbb{Z}^+$,
\[d_k(I,n)=\left(1+O\left(n^{-0.49}\right)\right)\binom{n}{t}f_k(n-t)\int_0^1\varphi_k(x)\int_{0}^{x}\sum_{\tau\in \mathcal{D}_k(I,t)}\Phi_{\tau(t)}^t(y)\hspace{0.5mm}\mathrm{d}y\hspace{0.5mm}\mathrm{d}x.\]
\end{theorem}

From Theorem~\ref{thm:precdasy} and Theorem~\ref{thm:nasy}, proving Theorem~\ref{thm:dasy} is easy algebra.

\begin{proof}[Proof of Theorem~\ref{thm:dasy}] 
We start with \[d_k(I,n)=\left(1+O\left(n^{-0.49}\right)\right)\binom{n}{t}f_k(n-t)\int_0^1\varphi_k(x)\int_{0}^{x}\sum_{\tau\in \mathcal{D}_k(I,t)}\Phi_{\tau(t)}^t(y)\hspace{0.5mm}\mathrm{d}y\hspace{0.5mm}\mathrm{d}x.\]
Plugging in $\binom{n}{t}=(1+O(1/n))\frac{n^t}{t!}$ and Theorem~\ref{thm:nasy}, we get
\[d_k(I,n)=\left(1+O\left(n^{-0.49}\right)\right)\left(1+O(1/n)\right)\frac{n^t}{t!}c_k r_k^{n-t}(n-t)!\left(1+O(\gamma_k^n)\right)\int_0^1\varphi_k(x)\int_{0}^{x}\sum_{\tau\in \mathcal{D}_k(I,t)}\Phi_{\tau(t)}^t(y)\hspace{0.5mm}\mathrm{d}y\hspace{0.5mm}\mathrm{d}x.\]
Only keeping the dominating error term and then using $n!=(n-t)!n^t\left(1+O(1/n)\right)$, and then again keeping the dominating error term, we further get
\[d_k(I,n)=\left(1+O\left(n^{-0.49}\right)\right)c_k n!r_k^{n}\frac{1}{t! r_k^t}\int_0^1\varphi_k(x)\int_{0}^{x}\sum_{\tau\in \mathcal{D}_k(I,t)}\Phi_{\tau(t)}^t(y)\hspace{0.5mm}\mathrm{d}y\hspace{0.5mm}\mathrm{d}x.\]
Using Theorem~\ref{thm:nasy} and then only keeping the dominating error term once more, we arrive at
\[d_k(I,n)=\left(1+O\left(n^{-0.49}\right)\right)f_k(n)\frac{1}{t! r_k^t}\int_0^1\varphi_k(x)\int_{0}^{x}\sum_{\tau\in \mathcal{D}_k(I,t)}\Phi_{\tau(t)}^t(y)\hspace{0.5mm}\mathrm{d}y\hspace{0.5mm}\mathrm{d}x.\]

This is what we sought to prove, with
\[c_{I,k}=\frac{1}{t! r_k^t}\int_0^1\varphi_k(x)\int_{0}^{x}\sum_{\tau\in \mathcal{D}_k(I,t)}\Phi_{\tau(t)}^t(y)\hspace{0.5mm}\mathrm{d}y\hspace{0.5mm}\mathrm{d}x.\]

\end{proof}

\subsection{A somewhat simpler formula for $c_{I,k}$}
We can write down something simpler for $c_{I,k}$. Namely, we will prove the following.
\begin{proposition}\label{prop:niceform}
Let $k\geq 3$, let $I\subseteq \mathbb{Z}^+$ be a finite set $t=\max(I)+k-1$, and let $r_k\in \mathbb{R}$ be as given in Theorem~\ref{thm:nasy}. Then
\[c_{I,k}=\frac{1}{t! r_k^t}\int_0^1\varphi_k(x)\int_{0}^{x}\sum_{s=1}^t d_k(r_t(I),t+1-s,t)\Phi_{s}^t(y)\hspace{0.5mm}\mathrm{d}y\hspace{0.5mm}\mathrm{d}x.\]
\end{proposition}

This might appear more complicated than the previous expression for $c_{I,k}$, but we have replaced a sum of possibly up to $t!$ terms with a sum of just $t$ terms which can be understood reasonably well (more on this after the proof).

\begin{proof}
Consider
\[\sum_{\tau\in \mathcal{D}_k(I,t)}\Phi_{\tau(t)}^t(y).\]
Let us pick some $s$ and count the number of terms with $\tau(t)=s$. In other words, we are counting the number of permutations in $\mathcal{D}_k(I,t)$ that end with $s$. Under reverse-complementation, $\mathcal{D}_k(I,t)$ bijects with $\mathcal{D}_k(r_t(I),t)$ (see Subsection~\ref{subsec:general} for a reminder on this notation). Under the same map, the subset of $\mathcal{D}_k(I,t)$ of permutations ending with $s$ bijects with $\mathcal{D}_k(r_t(I),t+1-s,t)$. This gives the coefficients in our regrouped sum:
\[\sum_{\tau\in \mathcal{D}_k(I,t)}\Phi_{\tau(t)}^t(y)=\sum_{s=1}^t d_k(r_t(I),t+1-s,t)\Phi_{s}^t(y).\]
Making this replacement in 
\[c_{I,k}=\frac{1}{t! r_k^t}\int_0^1\varphi_k(x)\int_{0}^{x}\sum_{\tau\in \mathcal{D}_k(I,t)}\Phi_{\tau(t)}^t(y)\hspace{0.5mm}\mathrm{d}y\hspace{0.5mm}\mathrm{d}x\]
gives us the desired result.
\end{proof}

We finish this subsection by mentioning that the coefficients $d_k(r_t(I),t+1-s,t)$ can be found using a dynamic programming approach similar to that in Subsection~\ref{subsec:genfast}. This makes $c_{I,k}$ efficiently computable, assuming one can efficiently take $n$ numerical integrals.


\subsection{Finitely many cases of asymptotic down-up-down-up}

Zhu \cite{zhu2019enumerating} made the following conjecture.

\begin{conjecture}[Down-Up-Down-Up Conjecture \cite{zhu2019enumerating}]\label{conj:dudu}

For any $n\in \mathbb{Z}^+$, 
\[d_3(\{1\},n)>d_3(\{2\},n)<d_3(\{3\},n)>d_3(\{4\},n)<\cdots,\]
where the sequence goes up to $\lceil n/2\rceil$.
\end{conjecture}

By bounding $c_{\{1\},3}, c_{\{2\},3}, \ldots$ using numerical integration, we can prove the following partial result towards the Down-Up-Down-Up Conjecture. It is partial in the sense that it only proves the conjectured inequality between $d_3(\{i\},n)$ and $d_3(\{i+1\},n)$ for a finite set of $i$, and also in the sense that our result is only for large enough $n$.

\begin{theorem}
There is $N\in \mathbb{Z}^+$ such that for all $n\geq N$, $d_3(\{1\},n)>d_3(\{2\},n)<d_3(\{3\},n)>d_3(\{4\},n)$.
\end{theorem}
\begin{proof}
Using numerical integration software to evaluate the expression for $c_{I,k}$ in Proposition~\ref{prop:niceform}, we found that 
\[\lim_{n\to \infty} \frac{d(\{1\},n)}{d(\{2\},n)}\approx 1.132101, \hspace{5mm} \lim_{n\to \infty} \frac{d(\{2\},n)}{d(\{3\},n)}\approx 0.826993, \hspace{5mm} \lim_{n\to \infty}\frac{d(\{3\},n)}{d(\{4\},n)}\approx 1.043244.\]
\end{proof}


\section{Asymptotic independence of the first and last element of a $k$-descent-avoiding permutation}\label{sec:joint}
Theorem~\ref{thm:fmnasy} gives the asymptotic distribution of the first element of $k$-descent-avoiding permutations. Namely, this distribution is given by $\varphi_k(x)$. Upon reverse-complementing, we also get that the distribution of the last element is given by $\varphi_k(1-x)$. However, to show equidistribution (Theorem~\ref{thm:equistrong}) in the next section, we will want to make use of the joint distribution of the first and last element. It is clear that the first and last element cannot both be $m$. Other than that, the first and last element turn out to be asymptotically independent. In this section, we state and prove this -- the rigorous statement is Theorem~\ref{thm:joint}. The structure and content of this section closely resemble those of Section~\ref{sec:mnasy}. We start with a definition.

\begin{definition}
For $k,n,m_1,m_2\in \mathbb{Z}^+$ with $1\leq m_1,m_2\leq n$, we let $f_k(m_1,m_2,n)$ be the number of $k$-descent-avoiding permutations $w\in S_n$ such that $w(1)=m_1$ and $w(n)=m_2$.
\end{definition}

\subsection{A proposition on $f_k(m_1,m_2,n)$} We will now state a proposition which is an analog of Proposition~\ref{prop:fmn}, except instead of conditioning on just the value of the first element, we now simultaneously condition on values for the first and last element, $m_1$ and $m_2$. Instead of an equality, we now have to settle for two inequalities, but the lower and upper bound will be relatively close.

\begin{proposition}\label{prop:fm1m2}
For $k\geq 3$, $n,m_1,m_2\in \mathbb{Z}^+$ with $1\leq m_1\leq m_2\leq n$, we have the following inequalities

\[ \sum_{\substack{\ell\\1\leq \ell \leq m_1\\ \ell\equiv 1 \pmod{k}}}\binom{m_1-2}{\ell-1}f_k\left(\min(n-\ell,n+1-m_2),n-\ell\right)\]\[ -\sum_{\substack{\ell\\k\leq \ell \leq m_1\\ \ell\equiv 0 \pmod{k}}}   \binom{m_1-1}{\ell-1}f_k\left(\max(1,n-\ell+1-m_2),n-\ell\right)\]
\[\leq f_k(m_1,m_2,n)\leq\]
\[ \sum_{\substack{\ell\\1\leq \ell \leq m_1\\ \ell\equiv 1 \pmod{k}}}\binom{m_1-1}{\ell-1}f_k\left(\max(1,n-\ell+1-m_2),n-\ell\right)\]\[-\sum_{\substack{\ell\\k\leq \ell \leq m_1\\ \ell\equiv 0 \pmod{k}}}\binom{m_1-2}{\ell-1}f_k\left(\min(n-\ell,n+1-m_2),n-\ell\right)\]
\end{proposition}

\begin{proof}
We let $h_k(m_1,m_2,\ell,n)$ be the number of permutations $w\in S_n$ such that $w(1)=m_1$, $w(n)=m_2$, $w(1)>w(2)>\cdots>w(\ell)$, and the restriction of $w$ to the last $n-\ell$ indices contains no $k$-descents. We use the convention that for $\ell>n$, $h_k(m_1,m_2,\ell,n)=0$. Then by an argument essentially identical to that in the proof of Proposition~\ref{prop:fmn},
\[f_k(m_1,m_2,n)=\sum_{\substack{\ell\\1\leq \ell \leq m_1\\ \ell\equiv 1 \pmod{k}}} h_k\left(m_1,m_2,\ell,n\right)-\sum_{\substack{\ell\\k\leq \ell \leq m_1\\ \ell\equiv 0 \pmod{k}}} h_k\left(m_1,m_2,\ell,n\right)\]

We proceed to bound the terms. We claim that
\[\binom{m_1-2}{\ell-1}f_k\left(\min(n-\ell,n+1-m_2),n-\ell\right)\leq h_k\left(m_1,m_2,\ell,n\right)\]\[\leq \binom{m_1-1}{\ell-1}f_k\left(\max(1,n-\ell+1-m_2),n-\ell\right).\]
We first prove the lower bound. There are at least $\binom{m_1-2}{\ell-1}$ options for the initial $\ell$-element decreasing sequence (we let this be any choice of $\ell-1$ positive integers strictly less than $m_1$ and not equal to $m_2$), after which there are $f_k\left(n-\ell+1-p,n-\ell\right)$ options for the $k$-descent-avoiding restriction to the last $n-\ell$ elements, where $p$ is the value of the last element $m_2$ in the relative ordering of the last $n-\ell$ elements ($p$ depends on the choice of an $\ell$-element decreasing sequence). We have that $\max(1,m_2-\ell)\leq p\leq \min(n-\ell,m_2)$. Using this and the fact that $f_k(m,n)$ is non-increasing in $m$ (for constant $n$), we get 
\[f_k(\min(n-\ell,n+1-m_2),n-\ell)\leq f_k\left(n-\ell+1-p,n-\ell\right)\leq f_k\left(\max(1,n-\ell+1-m_2),n-k\right).\]
Multiplying the lower bound on the number of $\ell$-element initial sequences with the lower bound on the number of options for the relative ordering of the last $n-\ell$ elements gives the desired lower bound on $h_k\left(m_1,m_2,\ell,n\right)$. The proof of the upper bound is similar, except now there are at most $\binom{m_1-1}{\ell-1}$ options for the initial decreasing sequence, and each leaves at most $f_k\left(\max(1,n-\ell+1-m_2),n-k\right)$ options for the relative ordering of the last $n-\ell$ elements.

The statement of the proposition follows from plugging in our bounds on $h_k(m_1,m_2,\ell,n)$ into our expression for $f_k(m_1,m_2,n)$.
\end{proof}


\subsection{The asymptotic distribution of $f_k(m_1,m_2,n)$} We move on to stating the main asymptotic independence result.

\begin{theorem}\label{thm:joint}
Fix $k\geq 3$. For $n,m_1,m_2\in \mathbb{Z}^+$ with $1\leq m_1, m_2\leq n$ and $m_1\neq m_2$,
\[\frac{n^2 f_k(m_1,m_2,n)}{f_k(n)}=\varphi_k\left(\frac{m_1}{n}\right)\varphi_k\left(1-\frac{m_2}{n}\right)\left(1+O_k\left(n^{-0.49}\right)\right).\]
\end{theorem}

The reason we are calling this an asymptotic independence theorem is that it follows from this and Theorem~\ref{thm:fmnasy} that for $m_1\neq m_2$ and $w\in \mathcal{D}_k(\emptyset,n)$ chosen uniformly at random,
\[\mathbb{P}(w(1)=m_1,w(n)=m_2)\sim \mathbb{P}\left(w(1)=m_1\right)\cdot\mathbb{P}\left(w(n)=m_2\right).\]

The proof of Theorem~\ref{thm:joint} closely resembles the proof of Theorem~\ref{thm:fmnasy}, with Proposition~\ref{prop:fm1m2} in place of Proposition~\ref{prop:fmn} and Theorem~\ref{thm:fmnasy} complementing Theorem~\ref{thm:nasy}.

\begin{proof}[Proof of Theorem~\ref{thm:joint}] It suffices to show that $\frac{n^2 f_k(m_1,m_2,n)}{f_k(n)}$ is both lower and upper bounded by $\varphi_k\left(\frac{m_1}{n}\right)\varphi_k\left(1-\frac{m_2}{n}\right)\left(1+O_k\left(n^{-0.49}\right)\right)$. We start with the lower bound given by Proposition~\ref{prop:fm1m2},

\[\frac{n^2f_k(m_1,m_2,n)}{f_k(n)}\geq \]

\begin{align*}
    &\frac{n^2}{f_k(n)}\sum_{\substack{\ell\\1\leq \ell \leq m_1\\ \ell\equiv 1 \pmod{k}}}\binom{m_1-2}{\ell-1}f_k\left(\min(n-\ell,n+1-m_2),n-\ell\right)\\
    &-\frac{n^2}{f_k(n)}\sum_{\substack{\ell\\k\leq \ell \leq m_1\\ \ell\equiv 0 \pmod{k}}}   \binom{m_1-1}{\ell-1}f_k\left(\max(1,n-\ell+1-m_2),n-\ell\right)=(*)
\end{align*}

Using Theorem~\ref{thm:fmnasy}, we get

\begin{align*}
    (*)=&\frac{n^2}{f_k(n)}\sum_{\substack{\ell\\1\leq \ell \leq m_1\\ \ell\equiv 1 \pmod{k}}}\binom{m_1-2}{\ell-1}\varphi_k\left(\frac{\min(n-\ell,n+1-m_2)}{n-\ell}\right)\frac{f_k(n-\ell)}{n-\ell}\left(1+O\left(n^{-0.49}\right)\right)\\ &-\frac{n^2}{f_k(n)}\sum_{\substack{\ell\\k\leq \ell \leq m_1\\ \ell\equiv 0 \pmod{k}}}   \binom{m_1-1}{\ell-1}\varphi_k\left(\frac{\max(1,n-\ell+1-m_2)}{n-\ell}\right)\frac{f_k(n-\ell)}{n-\ell}\left(1+O\left(n^{-0.49}\right)\right).
\end{align*}

Using Theorem~\ref{thm:nasy}, we then get
\[=\sum_{\substack{\ell\\1\leq \ell \leq m_1\\ \ell\equiv 1 \pmod{k}}}\binom{m_1-2}{\ell-1}\varphi_k\left(\frac{\min(n-\ell,n+1-m_2)}{n-\ell}\right)\frac{n}{(n-1)(n-2)\cdots (n-\ell)r_k^{\ell}}\left(1+O\left(n^{-0.49}\right)\right)\]\[ - \sum_{\substack{\ell\\k\leq \ell \leq m_1\\ \ell\equiv 0 \pmod{k}}}   \binom{m_1-1}{\ell-1}\varphi_k\left(\frac{\max(1,n-\ell+1-m_2)}{n-\ell}\right)\frac{n}{(n-1)(n-2)\cdots (n-\ell)r_k^{\ell}}\left(1+O\left(n^{-0.49}\right)\right).\]
Now, by an argument like in the proof of Theorem~\ref{thm:fmnasy}, the terms with $\ell\geq \log n$ only contribute $O(1/n)$. We first consider the case $m_1\geq \sqrt{n}$, and focus on terms with $\ell\leq \log n$. We repeat some estimates from the proof of Theorem~\ref{thm:fmnasy}:
\[\binom{m_1-1}{\ell-1}=\frac{m_1^{\ell-1}}{(\ell-1)!}\left(1+O\left(\frac{\ell^2}{m_1}\right)\right),\]
\[\binom{m_1-2}{\ell-1}=\frac{m_1^{\ell-1}}{(\ell-1)!}\left(1+O\left(\frac{\ell^2}{m_1}\right)\right),\text{ and}\]
\[\frac{n}{(n-1)(n-2)\cdots(n-\ell)}=n^{-\ell+1}\left(1+O\left(\frac{\ell^2}{n}\right)\right).\]
Plugging these into the current expression for our lower bound, we get
\begin{align*}
    (*)=&O(1/n)+\frac{1}{r_k}\sum_{\substack{\ell\\1\leq \ell \leq \log n\\ \ell\equiv 1 \pmod{k}}}\varphi_k\left(\frac{\min(n-\ell,n+1-m_2)}{n-\ell}\right)\frac{\left(\frac{m_1}{n}\frac{1}{r_k}\right)^{\ell-1}}{(\ell-1)!}\left(1+O\left(n^{-0.49}\right)\right)\\ &- \frac{1}{r_k}\sum_{\substack{\ell\\k\leq \ell \leq \log n\\ \ell\equiv 0 \pmod{k}}}  \varphi_k\left(\frac{\max(1,n-\ell+1-m_2)}{n-\ell}\right)\frac{\left(\frac{m_1}{n}\frac{1}{r_k}\right)^{\ell-1}}{(\ell-1)!}\left(1+O\left(n^{-0.49}\right)\right).
\end{align*}


In $[0,1]$, $\varphi_k$ is bounded below by a constant greater than $0$, and the derivative of $\varphi_k$ is bounded as well, so $\varphi_k\left(\frac{\min(n-\ell,n+1-m_2)}{n-\ell}\right)=\varphi_k\left(1-\frac{m_2}{n}\right)\left(1+O\left(\frac{\ell}{n}\right)\right)$. We plug this in as well, getting
\[(*)=O(1/n)+\frac{1}{r_k}\sum_{\substack{\ell\\1\leq \ell \leq \log n\\ \ell\equiv 1 \pmod{k}}}\varphi_k\left(1-\frac{m_2}{n}\right)\frac{\left(\frac{m_1}{n}\frac{1}{r_k}\right)^{\ell-1}}{(\ell-1)!}\left(1+O\left(n^{-0.49}\right)\right)\]\[ - \frac{1}{r_k}\sum_{\substack{\ell\\k\leq \ell \leq \log n\\ \ell\equiv 0 \pmod{k}}}  \varphi_k\left(1-\frac{m_2}{n}\right)\frac{\left(\frac{m_1}{n}\frac{1}{r_k}\right)^{\ell-1}}{(\ell-1)!}\left(1+O\left(n^{-0.49}\right)\right).\]

This can now be approximated by a corresponding infinite sum, just like in the proof of Theorem~\ref{thm:fmnasy}. Skipping the identical steps, we get that the expression for our lower bound becomes
\[(*)=O\left(n^{-0.49}\right)+\varphi_k\left(\frac{m_1}{n}\right)\varphi_k\left(1-\frac{m_2}{n}\right).\]
Now, we can get this to our desired form by noting that $\varphi_k$ is lower bounded by a positive constant,
\[(*)=\varphi_k\left(\frac{m_1}{n}\right)\varphi_k\left(1-\frac{m_2}{n}\right)\left(1+O\left(n^{-0.49}\right)\right).\]
So far, we are only done proving the desired lower bound for $m_1\geq \sqrt{n}$. We now prove the desired lower bound for $m_1\leq \sqrt{n}$. We come back to the lower bound
\[\sum_{\substack{\ell\\1\leq \ell \leq m_1\\ \ell\equiv 1 \pmod{k}}}\binom{m_1-2}{\ell-1}\varphi_k\left(\frac{\min(n-\ell,n+1-m_2)}{n-\ell}\right)\frac{n}{(n-1)(n-2)\cdots (n-\ell)r_k^{\ell}}\left(1+O\left(n^{-0.49}\right)\right)\]\[ - \sum_{\substack{\ell\\k\leq \ell \leq m_1\\ \ell\equiv 0 \pmod{k}}}   \binom{m_1-1}{\ell-1}\varphi_k\left(\frac{\max(1,n-\ell+1-m_2)}{n-\ell}\right)\frac{n}{(n-1)(n-2)\cdots (n-\ell)r_k^{\ell}}\left(1+O\left(n^{-0.49}\right)\right).\]
Consider just the first sum; it can be rewritten as
\[\sum_{\substack{\ell\\1\leq \ell \leq m_1\\ \ell\equiv 1 \pmod{k}}}\frac{1}{(\ell-1)!}\varphi_k\left(\frac{\min(n-\ell,n+1-m_2)}{n-\ell}\right)\frac{(m_1-2)(m_1-3)\ldots(m_1-\ell)}{(n-2)\cdots (n-\ell)r_k^{\ell}}\left(1+O\left(n^{-0.49}\right)\right).\]
Since $\frac{m_1-i}{n-i}\leq \frac{m_1}{n}\leq n^{-0.5}$, all terms after the first one can be upper bounded by a geometric series with first term $O(n^{-k/2})$ and ratio $O(n^{-k/2})$. The entire second sum can be upper bounded by a similar geometric series. All in all, this gives that the contribution of all terms other than the $\ell=1$ term is $O(1/n)$. Skipping a few steps again, we arrive at a lower bound of
\[\frac{1}{r_k}\varphi_k\left(1-\frac{m_2}{n}\right)+O\left(n^{-0.49}\right).\]
Since $\varphi_k$ has bounded derivative and $\varphi_k(0)=\frac{1}{r_k}$, this is
\[\varphi_k\left(\frac{m_1}{n}\right)\varphi_k\left(1-\frac{m_2}{n}\right)+O\left(n^{-0.49}\right)+O\left(n^{-0.5}\right).\]
As before, the desired lower bound follows using the fact that there is a positive lower bound on $\varphi_k$:
\[\varphi_k\left(\frac{m_1}{n}\right)\varphi_k\left(1-\frac{m_2}{n}\right)+O\left(n^{-0.49}\right)=\varphi_k\left(\frac{m_1}{n}\right)\varphi_k\left(1-\frac{m_2}{n}\right)\left(1+O\left(n^{-0.49}\right)\right).\]

The argument for the upper bound is completely analogous, and will be skipped. This completes the proof.
\end{proof}


\section{Asymptotic equidistribution}\label{sec:equi}
\subsection{Concentration for discrete order statistics: a tale of more than one Chebyshev}
    In this subsection, we prove a concentration result for certain discrete order statistics. Recall that according to Lemma~\ref{lem:os}, if we are choosing a subset of $t$ points from $[n]$, and $t$ stays constant while $n\to \infty$, then the $s$th largest point behaves just like a uniform $[0,1]$ $s$th order statistic. The main takeaway from the next lemma is that if $n\to \infty$ and $t(n)\to \infty$ as well, then the $s$th order statistic is quite close to being constant (that is, it is concentrated).

\begin{proposition}\label{prop:conc}
Let $s,t,n\in \mathbb{Z}^+$ with $1\leq s\leq t\leq n$. Let $\mathcal{Y}$ be a $t$-element subset of $[n]$ chosen uniformly from all $t$-element subsets of $[n]$. Let $Y_s$ be the $s$th largest element of $\mathcal{Y}$. Then we have the following:
\begin{enumerate}[label=(\arabic*)]
    \item $\mathbb{E}[Y_s]=\frac{s}{t+1}(n+1)$.
    \item $\Var(Y_s)\leq \frac{n^2}{t}$.
    \item $\mathbb{P}(|Y_s-\frac{s}{t+1}(n+1)|\leq \frac{n}{t^{1/3}})\geq 1-t^{-1/3}$.
\end{enumerate}
\end{proposition}

Our proof of the first part of Proposition~\ref{prop:conc} is inspired by the following argument which I learned in MIT's Fall 2017 Putnam seminar (18.A34), taught by Yufei Zhao. Let $Z_1, Z_2,\ldots, Z_t$ be independent random variables, each chosen uniformly from the (continuous) interval $[0,1]$. Defining $Z=\min(Z_1,\ldots, Z_t)$, our goal is to show that $\mathbb{E}[Z]=\frac{1}{t+1}$. We let $Z'_0,Z'_1,\ldots, Z'_t$ be independent random variables, each chosen uniformly on a circle of length $1$. Note that if we cut the circle at $Z_0'$, i.e. making it a segment $[0,1]$ starting at $Z_0'$ and oriented according to some chosen orientation on the circle, the (joint) distribution of $Z_1', \ldots, Z_t'$ on this segment is the same as the distribution of $Z_1,Z_2,\ldots,Z_t$. Under this identification, the distance between $Z_0'$ and the next $Z_i$ (in the direction given by this orientation on the circle) is equal to $Z$. For $0\leq i\leq t$, we define the random variable $D_i$ to be the distance between $Z_i'$ and the next $Z_j'$ (according to the same chosen orientation on the circle). Note that since the setup is symmetric under relabeling variables, $D_i$ and $D_j$ are identically distributed, so in particular, $\mathbb{E}[D_i]=\mathbb{E}[D_j]$. Going around the circle from some point $Z_i$, we note that the segment from each $Z_j$ to the next point is traversed exactly once, so the total distance is $1=D_0+D_1+\cdots+D_t$. Hence, $1=\mathbb{E}[D_0+\cdots+D_t]$. Together with linearity of expectation and our previous observation that expectations of any $D_i$ and $D_j$ are equal, this implies that $1=(t+1)\mathbb{E}[D_0]$, giving that $\mathbb{E}[D_0]=\frac{1}{t+1}$. One can extend this to the expectation of the $s$th largest $Z_i$ being $\frac{s}{t+1}$ by replacing the $D_i$ in this argument with the distance between the $\ell$th and $(\ell+1)$th $Z_j$ coming after $Z_i$ on the circle, and summing over $\ell=0,\ldots,s-1$ at the end.

In our case, we can start with $n+1$ points on a circle, choose a subset of $t+1$ of these, and argue exactly as in the continuous case, getting that the expected size of any gap is $\frac{n+1}{t+1}$. For the sake of variety, we give a different short argument along similar lines.

\begin{proof}[Proof of Proposition~\ref{prop:conc} (1)] With $\mathcal{Y}$ as in the proposition statement, we let $Y_1, \ldots, Y_t$ be the elements of $\mathcal{Y}$ in order, and we define the random variables $D_1=Y_1, D_2=Y_2-Y_1, \ldots, D_{t}=Y_t-Y_{t-1}, D_{t+1}=(n+1)-Y_t$. Note that for any $D_1,\ldots, D_{t+1}$ with all $D_i\geq 1$ and $D_1+\cdots+D_{t+1}=n+1$, there is a unique set $\mathcal{Y}$, namely the one given by $Y_1=D_1, Y_2=D_1+D_2,\ldots, Y_t=D_1+\cdots+D_t$. So we get the same distribution for $\mathcal{Y}$ if we choose it by picking a uniformly random sequence of $(t+1)$ positive integers summing to $n+1$, $D_1,\ldots,D_{t+1}$, and then taking the corresponding $\mathcal{Y}$. But in this latter formulation in terms of $D_1,\ldots, D_{t+1}$, it is clear that any $D_i$ and $D_j$ are identically distributed. So $n+1=\mathbb{E}[D_1+\cdots+D_{t+1}]=(t+1)\mathbb{E}[D_i]$, from where $\mathbb{E}[D_i]=\frac{n+1}{t+1}$. Hence, $\mathbb{E}[Y_s]=\mathbb{E}[D_1+\cdots+D_s]=s\frac{n+1}{t+1}$, which is what we wanted to show.
\end{proof}

For the second part, we use a method for finding certain binomial coefficient sums, along with some other tricks. Here is one of the other tricks, which will be useful for proving negative correlation of $D_i,D_j$.

\begin{lemma}[A partial Chebyshev's sum inequality]\label{lem:cheb}
Suppose we have a weakly increasing sequence $a_1\leq \cdots \leq  a_n$ and two weakly decreasing sequences $b_1 \geq \cdots \geq b_n\geq 0$ and $c_1\geq \cdots\geq c_n$. Then \[\left(\sum_i a_i b_i c_i\right)\left(\sum_i b_i\right)\leq \left(\sum_i a_i b_i\right)\left(\sum_i b_i c_i \right).\]
Or equivalently,
\[\sum_i a_i b_i c_i \leq \sum_i a_i b_i \frac{\sum_i b_i c_i}{\sum_i b_i}.\]
\end{lemma}
Let us first briefly explain why we are calling this a partial Chebyshev's sum inequality. For the sequences $a_1\leq \cdots \leq a_n$ and $b_1 c_1 \geq \cdots \geq b_n c_n$, Chebyshev's sum inequality gives the following:
\[\sum_i a_i b_i c_i \leq \sum_i a_i \frac{\sum_i b_i c_i}{n}.\]
In words, the weakly increasing sequence $a_i$ summed against the non-increasing sequence $b_i c_i$ is less or equal to $a_i$ summed against the average of the sequence $b_i c_i$. We can think of the second version of the inequality in Lemma~\ref{lem:cheb} as saying that $a_i$ summed against $b_i c_i$ is less or equal to $a_i$ summed against $b_i \frac{\sum_i b_i c_i}{\sum_i b_i}$, which is the unique sequence $b_i c$ with the same sum as $b_i c_i$. In other words, there is a sort of partial averaging (only $c_i$ is averaged out) of the sequence $b_i c_i$.

With this is mind, the most intuitive proof might be an inductive mass redistribution argument, but we instead give a short algebraic proof in the spirit of algebraic proofs of the rearrangement inequality and Chebyshev's sum inequality.

\begin{proof}[Proof of Lemma~\ref{lem:cheb}]
Note that for any $i,j\in [n]$, we have $0\leq b_i b_j (a_j-a_i)(c_i-c_j)$. Expanding out, we get
\[a_i b_i c_i b_j+a_j b_j c_j b_i\leq a_i b_i b_j c_j +a_j b_j b_i c_i.\]
Summing up over all pairs $1\leq i\leq j\leq n$, we get the desired inequality.
\end{proof}

As another remark, note that the special case $b_i=1$ is Chebyshev's sum inequality. We proceed to give the proof of part (2) of our concentration result.

\begin{proof}[Proof of Proposition~\ref{prop:conc} (2)]
\[\Var(Y_s)=\Var(D_1+\cdots+D_s)=\sum_{1\leq i\leq s}\Var(D_i)+2\sum_{i<j\leq s}\Cov(D_i,D_j).\]
Since the joint distribution of $D_i$ is symmetric in $i$, we can rewrite the above as
\[\Var(Y_s)=s\Var(D_1)+s(s-1)\Cov(D_1,D_2).\]

We first consider $\Var(D_1)=\mathbb{E}[D_1^2]-\mathbb{E}[D_1]^2.$ From the proof of part (1), we have that 
\[\mathbb{E}[D_1]^2=\left(\frac{n+1}{t+1}\right)^2.\]
As for $\mathbb{E}[D_1^2]$, we begin by writing down an explicit expression:
\[\mathbb{E}[D_1^2]=\sum_{i=1}^{n}i^2\mathbb{P}(D_1=i)=\sum_{i=1}^n i^2 \frac{\binom{n-i}{t-1}}{\binom{n}{t}}=\frac{1}{\binom{n}{t}}\sum_{i=1}^n i^2\binom{n-i}{t-1}.\]
We now write $i^2$ in terms of a suitable ``basis'' to replace the above sum with some simpler binomial coefficient sums. Namely, we use the ``basis'' of parameters that are not summed over, as well as terms of the form $(n-i+1)(n-i+2)\cdots (n-i+\ell)$, motivated by the fact that products of such terms and binomial coefficients are nice. We write $i^2=(n-i+1)(n-i+2)-(2n+3)(n-i+1)+(n+1)^2$. Plugging this into the sum, we get
\[\mathbb{E}[D_1^2]=\frac{1}{\binom{n}{t}}\left(t(t+1)\sum_{i=1}^n \binom{n-i+2}{t+1}-(2n+3)t\sum_{i=1}^n\binom{n-i+1}{t}+(n+1)^2\sum_{i=1}^n\binom{n-i}{t-1}\right).\]
We now use the hockey-stick identity $\sum_{i=r}^\ell \binom{i}{r}=\binom{\ell+1}{r+1}$ to find each sum.
\[\mathbb{E}[D_1^2]=\frac{1}{\binom{n}{t}}\left(t(t+1)\binom{n+2}{t+2}-(2n+3)t\binom{n+1}{t+1}+(n+1)^2\binom{n}{t}\right).\]
Expanding the binomial coefficients and canceling terms,
\[\mathbb{E}[D_1^2]=(n+2)(n+1)\frac{t}{t+2}-(2n+3)(n+1)\frac{t}{t+1}+(n+1)^2.\]
With some algebra, we get
\[\mathbb{E}[D_1^2]=\frac{(n+1)(2n-t+2)}{(t+1)(t+2)}.\]

From our expressions for $\mathbb{E}[D_1]^2$ and $\mathbb{E}[D_1^2]$, some more algebra gives 
\[\Var(D_1)=\frac{t(n+1)(n-t)}{(t+1)^2(t+2)}\leq \frac{n^2}{t^2}.\]

We will next show that $D_1,D_2$ are negatively correlated, i.e. that $\Cov(D_1,D_2)\leq 0$. One could do this using the method for finding binomial coefficient sums outlined above, but we instead opt to give two other proofs.

The first proof of $\Cov(D_1,D_2)\leq 0$ uses Lemma~\ref{lem:cheb}. As $\Cov(D_1,D_2)=\mathbb{E}[D_1 D_2]-\mathbb{E}[D_1]\mathbb{E}[D_2]$, it is equivalent to show that $\mathbb{E}[D_1 D_2]\leq \mathbb{E}[D_1]\mathbb{E}[D_2]$. We have
\[\mathbb{E}[D_1 D_2]=\sum_{i=1}^n i\cdot \mathbb{P}(D_1=i)\cdot \mathbb{E}[D_2|D_1=i]\]
and
\[\mathbb{E}[D_1]\mathbb{E}[D_2]=\sum_{i=1}^n i \cdot \mathbb{P}[D_1=i]\cdot \mathbb{E}[D_2].\]
We start by noting that for $i=1,\ldots,n$, the sequence $a_i:=i$ is weakly increasing, the sequence $b_i:=\mathbb{P}(D_1=i)=\binom{n-i}{t-1}$ is weakly decreasing, and the sequence $c_i=\mathbb{E}[D_2|D_1=i]=\frac{n-i+1}{t}$ is also weakly decreasing.

Next, note that

\[\sum_i \mathbb{P}(D_1=i)\cdot\mathbb{E}[D_2|D_1=i]=\mathbb{E}[\mathbb{E}[D_2|D_1]=\mathbb{E}[D_2]=\sum_{i}\mathbb{P}[D_1=i]\cdot\mathbb{E}[D_2]\]

so
\[\mathbb{E}[D_2]=\frac{\sum_i \mathbb{P}(D_1=i)\cdot\mathbb{E}[D_2|D_1=i]}{\sum_{i}\mathbb{P}(D_1=i)}.\]
Translating this into $a_i,b_i,c_i$, we get
\[\mathbb{E}[D_2]=\frac{\sum_i b_i c_i}{\sum_i b_i}.\]
We can now express both $\mathbb{E}[D_1 D_2]$ and $\mathbb{E}[D_1]\mathbb{E}[D_2]$ in terms of $a_i,b_i,c_i$. Namely, the inequality we wish to prove is
\[\sum_i a_i b_i c_i\leq \sum_i a_i b_i \frac{\sum_i b_i c_i}{\sum_i b_i}.\]
But this inequality is precisely given by Lemma~\ref{lem:cheb}. 

The second proof of $\Cov(D_1,D_2)\leq 0$ uses the following trick. Note that $D_1+\cdots+D_{t+1}=n+1$ is a constant, so
\[0=\Var(D_1+\ldots+D_{t+1})=(t+1)\Var(D_1)+t(t+1)\Cov(D_1,D_2),\]
from which
\[\Cov(D_1,D_2)=\frac{-\Var(D_1)}{t}\leq 0.\]

Having shown that $\Cov(D_1,D_2)\leq 0$, we now finally come back to 
\[\Var(Y_s)=s\Var(D_1)+s(s-1)\Cov(D_1,D_2)\leq s\Var(D_1)\leq t\Var(D_1)\leq t\frac{n^2}{t^2}=\frac{n^2}{t}.\]
This is what we wanted to show.
\end{proof}


As a side remark, we note that using the second approach, one can also write down a precise equation for the covariance of $D_1$ and $D_2$, and hence the variance of $Y_s$:
\[\Cov(D_1,D_2)=\frac{-\Var(D_1)}{t}=-\frac{(n+1)(n-t)}{(t+1)^2(t+2)},\]
and hence

\begin{align*}
    \Var(Y_s)&=s\Var(D_1)+s(s-1)\Cov(D_1,D_2)=s\left(1-\frac{s-1}{t}\right)\Var(D_1)\\
    &=\frac{s(t-s+1)(n+1)(n-t)}{(t+1)^2(t+2)}.
\end{align*}

Note that $Y_s$ is minimal at $s=\lfloor (t+1)/2 \rfloor$ and $s=\lceil (t+1)/2\rceil$, which is what one might intuitively expect before making any calculations.

The third part of Proposition~\ref{prop:conc} follows from the previous two parts using a more famous Chebyshev's inequality, which says that a variance bound implies concentration. 

\begin{theorem}[Chebyshev's inequality] Let $X$ be a bounded real-valued random variable. Then for any real number $q>0$,
\[\mathbb{P}\left(|X-\mathbb{E}[X]|\geq q\sqrt{\Var(X)}\right)\leq q^{-2}.\]
\end{theorem}

\begin{proof}[Proof of Proposition~\ref{prop:conc} (3)] This is just Chebyshev's inequality with $q=t^{1/6}$, followed by the fact that $\sqrt{\Var(X)}\leq \frac{n}{\sqrt{t}}$ by part (2).
\end{proof}

\subsection{The equidistribution theorem} In this subsection, we prove the following equidistribution theorem.

\begin{theorem}\label{thm:equistrong}
Fix $k\geq 3$ and $a\in \mathbb{Z}^+$. There is a real constant $C_{k,a}>0$ such that for $n\in \mathbb{Z}^+$ and $I=\{i_1,i_2,\ldots, i_a\}\subseteq [n]$ with $i_1\geq \sqrt{n}$, $i_{j+1}-i_j\geq \sqrt{n}$ for all $j$ with $1\leq j\leq a-1$, and $n-i_a\geq \sqrt{n}$, we have
\[d_k(I,n)=C_{k,a} f_k(n)\left(1+O_{k,a}\left(n^{-1/6}\right)\right).\]
\end{theorem}

Calling this an equidistribution theorem is motivated by the fact that the constant $C_{k,a}$ only depends on $k$ and $a$ (and not on $I$, as long as $|I|=a$). In particular, we have the following immediate corollary.
\begin{theorem}\label{thm:equi}
Fix $k\geq 3$ and $a\in \mathbb{Z}^+$. Let $n\in \mathbb{Z}^+$, $I_1,I_2\subseteq [n]$ with $|I_1|=|I_2|=a$, and no two elements of $I_1$ being closer to each other or $1$ or $n$ than $\sqrt{n}$, and similarly for $I_2$. Then 
\[\frac{d_k(I_1,n)}{d_k(I_2,n)}=1+O_{k,a}\left(n^{-1/6}\right).\]
\end{theorem}

The proof of Theorem~\ref{thm:equistrong} uses a similar random permutation framework as Section~\ref{sec:dasy}.
\begin{proof}[Proof of Theorem~\ref{thm:equistrong}]
Consider some $w\in \mathcal{D}_k(I,n)$. We define $v_1$ to be the restriction of $w$ to the first $t_1:=i_1-1$ indices. Then $v_1\in \mathcal{D}_k(\emptyset,t_1)$. The next $k$ indices will be a $k$-descent. We define $v_2$ to be the restriction of $w$ to the indices after that and before the next $k$-descent at $i_2$, i.e. the $v_2$ is the restriction of $w$ to the $t_2:=i_2-i_1-k$ indices $i_1+k,i_1+k+1,\ldots,i_2-1$. In general, for any $j$ with $2\leq j\leq a$, we define $v_j$ to be the restriction of $w$ to the indices between the $(j-1)$th and $j$th $k$-descent -- these are the $t_j:=i_j-i_{j-1}-k$ indices $i_{j-1}+k,i_{j-1}+k+1,\ldots, i_j-1$. We let $v_{a+1}$ be the restriction of $w$ to the indices after the $a$th (and final) $k$-descent -- these are the $t_{a+1}:=n-i_a-k+1$ indices $i_a+k,i_a+k+1,\ldots, n$. Note that for all $i\in [a+1]$, we have $v_i\in \mathcal{D}_k(\emptyset,t_i)$. 

We now switch gears to constructing a random permutation $w\in S_n$. We determine a permutation $w$ by randomly choosing various relative orderings and values. Specifically, the random process is the following.
\begin{enumerate}[label=\arabic*.]
    \item Deterministically fix the relative orderings of the prescribed $k$-descents. That is, for any $j\in [a]$, we fix the relative ordering of the elements $i_j,i_j+1,\ldots,i_j+k-1$ to be $k,k-1,\ldots,1$.
    \item For $t_1,t_2,\ldots,t_{a+1}$, independently pick a uniformly random $v_i\in \mathcal{D}(\emptyset,t_i)$. These will be the relative orderings of each of the $a+1$ blocks formed of indices not involved in a $k$-descent. That is, $v_1$ will be the restriction of $w$ to the indices $1,\ldots, i_1-1$, $v_2$ will be the restriction of $w$ to the indices $i_1+k, i_1+k+1,\ldots,i_2-1$, and so on, until $v_{a+1}$ being the restriction of $w$ to the indices $i_a+k, i_{a}+k+1,\ldots, n$.
    \item Pick a uniformly random partition of $[n-a k]$ into parts of sizes $t_1,t_2,\ldots, t_{a+1}$. These will be the sets of values of each of our $a+1$ blocks in the relative ordering of the union of their indices. (At this point, we have determined the restriction of $w$ to the union of the $a+1$ $k$-descent-avoiding blocks.
    \item (1) Pick a $k$-element subset of $[n-(a-1)k]$. These will be the values of the first $k$-descent in the relative ordering on the union of the $[n-a k]$ indices from the previous part and the $k$ indices $i_1,i_1+1,\ldots,i_1+k-1$. Then (2) pick a $k$-element subset of $[n-(a-2)k]$ to be the values of the second prescribed $k$-descent in the relative ordering where the indices $i_2,i_2+1,\ldots,i_2+k-1$ are unioned in as well. Continue so up until and including ($a$): pick a $k$-element subset of $[n]$ to be the values of the last prescribed $k$-descent at $i_a, i_a+1,\ldots i_a+k-1$ in the relative ordering of all $[n]$ indices. Having finished this, we have determined the permutation $w$.  
\end{enumerate}
Figure~\ref{fig:blocks} shows the $k$-descents, $k$-descent-avoiding blocks, and relative orderings in an example $w$.
\begin{figure}
    \centering
    \begin{tikzpicture}[scale=0.25]
        \node [align=left] at (-5,6.6) {\tiny{block $1$ of length $t_1=8$} \\ \tiny{and relative ordering $v_1=64278513$}};
        
        \node [align=left] at (8.8,8.4) {\tiny{block $2$ of length $t_2=14$} \\ \tiny{and relative ordering $v_2$}};
        
        \node [align=left] at (24,6.6) {\tiny{block $3$ of length $t_3=8$} \\ \tiny{and relative ordering $v_3=62785341$}};
        
        \node [align=left] at (35,9) {\tiny{block $4$ of length $t_4=5$} \\ \tiny{and relative ordering $v_4=42531$}};
        
		\node [draw=none] (0) at (-5.25, -0.25) {};
		\node [draw=none] (1) at (-4, -1.25) {};
		\node [draw=none] (2) at (-3, 3) {};
		\node [draw=none] (3) at (-2.25, 3.75) {};
		\node [draw=none] (4) at (-1.5, 1.25) {};
		\node [draw=none] (5) at (-1, -1.5) {};
		\node [draw=none] (6) at (0, -0.75) {};
		\node [draw=none] (7) at (0.75, 0.5) {};
		\node [draw=none] (8) at (1.5, -0.75) {};
		\node [draw=none] (9) at (2.25, -1.25) {};
		\node [draw=none] (10) at (3.25, -1.75) {};
		\node [draw=none] (11) at (4.25, -0.75) {};
		\node [draw=none] (12) at (4.75, 2.25) {};
		\node [draw=none] (13) at (5.75, 3.75) {};
		\node [draw=none] (14) at (7, 4.75) {};
		\node [draw=none] (15) at (6.75, -0.5) {};
		\node [draw=none] (16) at (5.5, 1.75) {};
		\node [draw=none] (17) at (7.75, -1.25) {};
		\node [draw=none] (18) at (8.75, 0.5) {};
		\node [draw=none] (19) at (9.5, 1.25) {};
		\node [draw=none] (20) at (10.5, -0.75) {};
		\node [draw=none] (21) at (11, 3.25) {};
		\node [draw=none] (22) at (12, -0.75) {};
		\node [draw=none] (23) at (13, -2.5) {};
		\node [draw=none] (24) at (13.5, 1.75) {};
		\node [draw=none] (25) at (14, -1.5) {};
		\node [draw=none] (26) at (14.75, 0) {};
		\node [draw=none] (27) at (15.5, -0.75) {};
		\node [draw=none] (28) at (16.25, -1.5) {};
		\node [draw=none] (29) at (17, -3.25) {};
		\node [draw=none] (30) at (18.5, 2.75) {};
		\node [circle,fill,scale=0.2] (31) at (19.75, -1.25) {};
		\node [draw=none] (32) at (20.5, 3.5) {};
		\node [draw=none] (33) at (21.75, 4) {};
		\node [draw=none] (34) at (22.75, 2) {};
		\node [draw=none] (35) at (23.5, 1) {};
		\node [draw=none] (36) at (24.25, 2) {};
		\node [draw=none] (37) at (25, -1.75) {};
		\node [draw=none] (38) at (26, 0) {};
		\node [draw=none] (39) at (27, -1.25) {};
		\node [draw=none] (40) at (27.75, -2.25) {};
		\node [draw=none] (41) at (28.5, -3.25) {};
		\node [draw=none] (42) at (29.25, 2.5) {};
		\node [draw=none] (43) at (30, 1.25) {};
		\node [draw=none] (44) at (30.75, 4) {};
		\node [draw=none] (45) at (31.5, 2.5) {};
		\node [draw=none] (46) at (34.25, -1.5) {};
		\node [draw=none] (48) at (-6.25, 1.75) {};
		\node [circle,fill,scale=0.2] (49) at (-6.25, 1.75) {};
		\node [circle,fill,scale=0.2] (50) at (-5.25, -0.25) {};
		\node [circle,fill,scale=0.2] (51) at (-4, -1.25) {};
		\node [circle,fill,scale=0.2] (52) at (-3, 3) {};
		\node [circle,fill,scale=0.2] (53) at (-2.25, 3.75) {};
		\node [circle,fill,scale=0.2] (54) at (-1.5, 1.25) {};
		\node [circle,fill,scale=0.2] (55) at (-1, -1.5) {};
		\node [circle,fill,scale=0.2] (56) at (0, -0.75) {};
		\node [circle,fill,scale=0.2] (57) at (0.75, 0.5) {};
		\node [circle,fill,scale=0.2] (58) at (1.5, -0.75) {};
		\node [circle,fill,scale=0.2] (59) at (2.25, -1.25) {};
		\node [circle,fill,scale=0.2] (60) at (3.25, -1.75) {};
		\node [circle,fill,scale=0.2] (61) at (4.25, -0.75) {};
		\node [circle,fill,scale=0.2] (62) at (4.75, 2.25) {};
		\node [circle,fill,scale=0.2] (63) at (5.5, 1.75) {};
		\node [circle,fill,scale=0.2] (64) at (5.75, 3.75) {};
		\node [circle,fill,scale=0.2] (65) at (6.75, -0.5) {};
		\node [circle,fill,scale=0.2] (66) at (7, 4.75) {};
		\node [circle,fill,scale=0.2] (67) at (7.75, -1.25) {};
		\node [circle,fill,scale=0.2] (68) at (8.75, 0.5) {};
		\node [circle,fill,scale=0.2] (69) at (9.5, 1.25) {};
		\node [circle,fill,scale=0.2] (70) at (10.5, -0.75) {};
		\node [circle,fill,scale=0.2] (71) at (11, 3.25) {};
		\node [circle,fill,scale=0.2] (72) at (12, -0.75) {};
		\node [circle,fill,scale=0.2] (73) at (13, -2.5) {};
		\node [circle,fill,scale=0.2] (74) at (13.5, 1.75) {};
		\node [circle,fill,scale=0.2] (75) at (14, -1.5) {};
		\node [circle,fill,scale=0.2] (76) at (14.75, 0) {};
		\node [circle,fill,scale=0.2] (77) at (15.5, -0.75) {};
		\node [circle,fill,scale=0.2] (78) at (16.25, -1.5) {};
		\node [circle,fill,scale=0.2] (79) at (17, -3.25) {};
		\node [circle,fill,scale=0.2] (80) at (18.5, 2.75) {};
		\node [circle,fill,scale=0.2] (81) at (20.5, 3.5) {};
		\node [circle,fill,scale=0.2] (82) at (21.75, 4) {};
		\node [circle,fill,scale=0.2] (83) at (22.75, 2) {};
		\node [circle,fill,scale=0.2] (84) at (23.5, 1) {};
		\node [circle,fill,scale=0.2] (85) at (24.25, 2) {};
		\node [circle,fill,scale=0.2] (86) at (25, -1.75) {};
		\node [circle,fill,scale=0.2] (87) at (26, 0) {};
		\node [circle,fill,scale=0.2] (88) at (27, -1.25) {};
		\node [circle,fill,scale=0.2] (89) at (27.75, -2.25) {};
		\node [circle,fill,scale=0.2] (90) at (28.5, -3.25) {};
		\node [circle,fill,scale=0.2] (91) at (29.25, 2.5) {};
		\node [circle,fill,scale=0.2] (92) at (30, 1.25) {};
		\node [circle,fill,scale=0.2] (93) at (30.75, 4) {};
		\node [circle,fill,scale=0.2] (94) at (31.5, 2.5) {};
		\node [circle,fill,scale=0.2] (95) at (34.25, -1.5) {};
		\node [circle,fill,scale=0.2] (96) at (19.75, -1.25) {};
		\node [draw=none] (97) at (-7, 4.75) {};
		\node [draw=none] (98) at (-7.25, -1.75) {};
		\node [draw=none] (99) at (0.5, -2) {};
		\node [draw=none] (100) at (0, 4.75) {};
		\node [draw=none] (101) at (4, -2.5) {};
		\node [draw=none] (102) at (14.25, -3) {};
		\node [draw=none] (103) at (14.25, 6) {};
		\node [draw=none] (104) at (3.75, 5.75) {};
		\node [draw=none] (105) at (17.75, -2) {};
		\node [draw=none] (106) at (17.25, 4.75) {};
		\node [draw=none] (107) at (25.5, -2) {};
		\node [draw=none] (108) at (25.25, 4.5) {};
		\node [draw=none] (109) at (28.25, 4) {};
		\node [draw=none] (110) at (29.25, -1.25) {};
		\node [draw=none] (111) at (34.5, -2.5) {};
		\node [draw=none] (112) at (34.5, 4.75) {};
		\draw (48.center) to (0.center);
		\draw (0.center) to (1.center);
		\draw (1.center) to (2.center);
		\draw (2.center) to (3.center);
		\draw (3.center) to (4.center);
		\draw (4.center) to (5.center);
		\draw (5.center) to (6.center);
		\draw (6.center) to (7.center);
		\draw (7.center) to (8.center);
		\draw (8.center) to (9.center);
		\draw (9.center) to (10.center);
		\draw (10.center) to (11.center);
		\draw (11.center) to (12.center);
		\draw (12.center) to (16.center);
		\draw (16.center) to (13.center);
		\draw (13.center) to (15.center);
		\draw (15.center) to (14.center);
		\draw (14.center) to (17.center);
		\draw (17.center) to (18.center);
		\draw (18.center) to (19.center);
		\draw (19.center) to (20.center);
		\draw (20.center) to (21.center);
		\draw (21.center) to (22.center);
		\draw (22.center) to (23.center);
		\draw (23.center) to (24.center);
		\draw (24.center) to (25.center);
		\draw (25.center) to (26.center);
		\draw (26.center) to (27.center);
		\draw (27.center) to (28.center);
		\draw (28.center) to (29.center);
		\draw (29.center) to (30.center);
		\draw (30.center) to (31);
		\draw (31) to (32.center);
		\draw (32.center) to (33.center);
		\draw (33.center) to (34.center);
		\draw (34.center) to (35.center);
		\draw (35.center) to (36.center);
		\draw (36.center) to (37.center);
		\draw (37.center) to (38.center);
		\draw (38.center) to (39.center);
		\draw (39.center) to (40.center);
		\draw (40.center) to (41.center);
		\draw (41.center) to (42.center);
		\draw (42.center) to (43.center);
		\draw (43.center) to (44.center);
		\draw (44.center) to (45.center);
		\draw (45.center) to (46.center);
		\draw [red] (97.center) to (100.center);
		\draw [red] (100.center) to (99.center);
		\draw [red] (99.center) to (98.center);
		\draw [red] (98.center) to (97.center);
		\draw [red] (104.center) to (101.center);
		\draw [red] (101.center) to (102.center);
		\draw [red] (102.center) to (103.center);
		\draw [red] (103.center) to (104.center);
		\draw [red] (106.center) to (105.center);
		\draw [red] (105.center) to (107.center);
		\draw [red] (107.center) to (108.center);
		\draw [red] (108.center) to (106.center);
		\draw [red] (109.center) to (110.center);
		\draw [red] (110.center) to (111.center);
		\draw [red] (111.center) to (112.center);
		\draw [red] (112.center) to (109.center);

\end{tikzpicture}

    \caption{$k=4$, $n=47$, and $w\in S_{47}$ is a permutation with $k$-descents starting at $i_1=9$, $i_2=27$, $i_3=39$. The number of $k$-descents is $a=3$, and the $3+1=4$ blocks avoiding $k$-descents are shown in red.}
    \label{fig:blocks}
\end{figure}
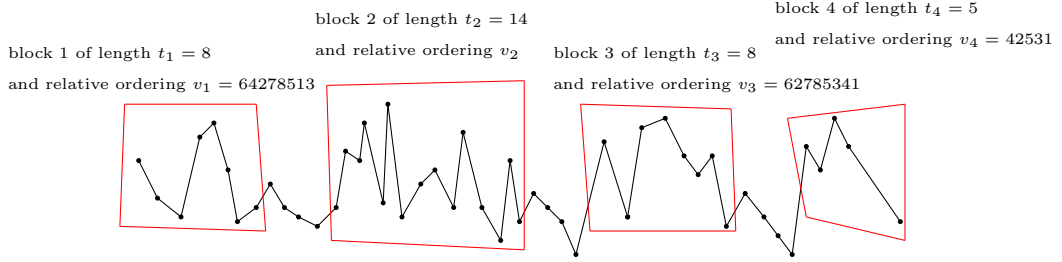

By what we observed in the first paragraph of the proof, each permutation $w\in \mathcal{D}_k(I,n)$ can be generated by this process (in a unique way). However, not every permutation generated by this is in $\mathcal{D}_k(I,n)$. In fact, a $w$ generated like this will be in $\mathcal{D}_k(I,n)$ if and only if each of the pre-determined $k$-descents is preceded and followed by an ascent. That is, for such $w$, \[w\in \mathcal{D}_k(I,n)\iff \forall j\in [a], w(i_j-1)<w(i_j),w(i_j+k-1)<w(i_j+k).\]
Let us call this event $\mathcal{E}_{I,n}$. Then, by counting the total number of permutations that can be constructed with the described procedure, we get
\[d_k(I,n)=\left(\prod_{j=1}^{a+1}f_k(t_j)\right)\frac{n!}{t_1!t_2!\cdots t_{\ell+1}!\left(k!\right)^a}\mathbb{P}(\mathcal{E}_{I,n}).\]

We claim that it now suffices to show that \[\mathbb{P}(\mathcal{E}_{I,n})=C_{k,a}\left(1+O\left(n^{-1/6}\right)\right),\] where $C_{k,a}$ is a constant only depending on $a$ and $k$ (and not on $I$). Indeed, if this is given, then using this and Theorem~\ref{thm:nasy} on the previous expression for $d_k(I,n)$ we get:

\begin{align*}
    \frac{d_k(I,n)}{f_k(n)}&=\frac{1}{f_k(n)}\left(\prod_{j=1}^{a+1}f_k(t_j)\right)\frac{n!}{t_1!t_2!\cdots t_{a+1}!\left(k!\right)^a}\mathbb{P}(\mathcal{E}_{I,n})\\
    &=\frac{c_k^a}{k!r_k^{ak}}C_{k,a}\left(1+O\left(n^{-1/6}\right)\right)
\end{align*}

We proceed to show that $\mathbb{P}(\mathcal{E}_{I,n})=C_{k,a}\left(1+O\left(n^{-1/6}\right)\right)$. Let us consider the state after completing step 2 of the above procedure. We claim that the joint distribution of the first and last elements in each relative ordering is asymptotically known. Namely, Theorem~\ref{thm:joint} together with independence of different blocks gives that for any sequence of integers \[m_{1,1},m_{2,1},m_{1,2},m_{2,2},m_{1,3},m_{2,3},\ldots, m_{1,a+1},m_{2,a+1}\]
such that for all $j\in [a+1]$,
\[m_{1,j}\neq m_{2,j} \hspace{3mm}\text{and}\hspace{3mm}1\leq m_{1,j},m_{2,j}\leq t_j,\]
we have
\[\mathbb{P}\left(v_1(1)=m_{1,1}, v_1(t_1)=m_{2,1},\ldots, v_{r+1}(1)=m_{1,a+1},v_{a+1}(t_{a+1})=m_{2,a+1}\right)=\]
\[=\frac{\varphi_k\left(\frac{m_{1,1}}{t_1}\right)\varphi_k\left(1-\frac{m_{2,1}}{t_1}\right)\cdots\varphi_k\left(\frac{m_{1,a+1}}{t_{a+1}}\right)\varphi_k\left(1-\frac{m_{2,a+1}}{t_{a+1}}\right)}{t_1^2\cdots t_{a+1}^2}\left(1+O\left(n^{-0.49}\right)\right).\]

We now fix $m_{1,j},m_{2,j}$ for all $j$ and consider step 3. We call the values assigned to the beginning and end of each block in step 3 respectively $\ell_{1,1},\ell_{2,1}, \ldots, \ell_{1,a+1},\ell_{2,a+1}$. By our concentration result Proposition~\ref{prop:conc} (3), the probability that $\frac{\ell_{1,1}}{n}$ is within $n^{-1/6}$ of $\frac{m_{1,1}}{t_1}$ is at least $1-n^{-1/6}$. By union-bounding, the probability that the analogous statement holds for all $\ell_{1,j}$ and $\ell_{2,j}$ simultaneously is at least $1-O_{k,a}\left(n^{-1/6}\right)$. The other case has probability at most $O\left(n^{-1/6}\right)$, which will be included in the error term later.

We now claim that there is a continuous function $\theta_k(x,y)\colon [0,1]^2\to [0,1]$, continuously differentiable on $[0,1]^2$, such that the following holds. We fix any sequence $\ell_{1,1},\ell_{2,1}, \ldots, \ell_{1,a+1},\ell_{2,a+1}$ to be the values of first and last elements of blocks assigned in step 3.  Then the probability that we get the desired ordering in step 4.(1), i.e. there is an ascent before and after the $k$-descent at $i_1$, is $\theta_k\left(\frac{\ell_{2,1}}{n},\frac{\ell_{1,2}}{n}\right)+O\left(n^{-0.5}\right)$.

The proof of this claim boils down to estimating binomial coefficients, and is similar to the proof of Lemma~\ref{lem:os}. Note that we are choosing $k$ elements from $[n-(a-1)k]$, for which there are a total of $\binom{n-(a-1)k}{k}$ options. We wish to count the number of options for which the largest of these $k$ elements is greater than the $\ell_{2,1}$th of the other $n-ak$ elements, which happens iff the largest is at least $\ell_{2,1}+k$; and simultaneously the smallest of these $k$ elements is at most the value of the $\ell_{1,2}$th of the other $n-ak$ elements, which happens iff the smallest is less than $\ell_{1,2}$. Motivated by this, let us now count the number of ways to pick a $k$-element subset of $[n]$ such that the largest element is greater than $\ell_2$ and the smallest element is at most $\ell_1$. The total number of choices of $k$ elements out of $[n]$ is $\binom{n}{k}$, out of which $\binom{\ell_2}{k}$ have largest element at most $\ell_2$, and $\binom{n-\ell_1}{k}$ have smallest element greater than $\ell_1$, with either $\binom{\ell_2-\ell_1}{k}$ or $0$ choices having both, depending on whether $\ell_2>\ell_1$. Hence, the number of ways to pick a $k$-element subset for which the largest element is greater than $\ell_2$ and the smallest element is at most $\ell_1$ is
$\binom{n}{k}-\binom{\ell_2}{k}-\binom{\ell_1}{k}+\mathbbm{1}_{\ell_2>\ell_1}\binom{\ell_2-\ell_1}{k}$. So the probability of such a choice is
\[1-\frac{\binom{\ell_2}{k}}{\binom{n}{k}}-\frac{\binom{\ell_1}{k}}{\binom{n}{k}}+\mathbbm{1}_{\ell_2>\ell_1}\frac{\binom{\ell_2-\ell_1}{k}}{\binom{n}{k}}\]
\[=1-\left(\frac{\ell_2}{n}\right)^k-\left(\frac{\ell_1}{n}\right)^k+\mathbbm{1}_{\ell_2>\ell_1}\left(\frac{\ell_2-\ell_1}{n}\right)^k+O\left(n^{-0.5}\right),\]
where the binomial coefficients are estimated like in the proof of Lemma~\ref{lem:os} by splitting into cases according to whether we are choosing at most $\sqrt{n}$ or more than $\sqrt{n}$ elements. Now we note that the last expression is $\theta_k\left(\frac{\ell_2}{n}, \frac{\ell_1}{n}\right)+O\left(n^{-0.5}\right)$, where
\[\theta_k(x,y)\colon [0,1]^2\to [0,1], \hspace{5mm} (x,y)\mapsto 1-x^k-y^k+\mathbbm{1}_{x>y}(x-y)^k.\]

Note that $\theta_k$ is continuously differentiable on $[0,1]^2$, so there is a bound on its derivative (over all points in $[0,1]^2$ and all directions). When calculating this probability, the total number of elements was $n$ instead of $n-ak$ as in the initial setup, $\ell_{2,1}+k$ was replaced with $\ell_2+1$, and $\ell_{1,2}$ was replaced with $\ell_1-1$. However, all of these are $O_{k,a}\left(\frac{1}{n}\right)$ changes to an input of $\theta_k$, so since $\theta_k$ has bounded derivative, these contribute a $O\left(\frac{1}{n}\right)$ change to the value of $\theta_k$. Hence, the probability that we get the desired ordering in step 4.(1), i.e., there is an ascent before and after the $k$-descent at $i_1$, is $\theta_k\left(\frac{\ell_{2,1}}{n},\frac{\ell_{1,2}}{n}\right)+O\left(n^{-0.5}\right)$. The existence of such $\theta_k$ is what we wanted to show. We now further claim that conditioned on any choice in 4.(1), the probability that we get the desired ascents in step 4.(2), i.e., that there is an ascent before and after the $k$-descent at $i_2$, is also $\theta_k\left(\frac{\ell_{2,2}}{n},\frac{\ell_{1,3}}{n}\right)+O_{k,a}\left(n^{-\beta}\right)$, and so on (conditioning on any previous choices) until 4.($a$). The proof of this further claim is almost exactly the same, with the only additional observations being that the total number of elements is at most $ak$ away from $n$ at any step, and that an index which has value $\ell_{u,j}$ after step 3. of the process will have value at least $\ell_{u,j}$ and at most $\ell_{u,j}+ak$ when we get to deciding the values of the $k$-descent adjacent to it in step 4. But there is a uniform $O_{k,a}\left(1/n\right)$ bounding these changes to inputs, so again since $\theta_k$ has bounded derivative, (conditional on any sequence of choices in steps 4.(1), $\ldots$, 4.(j-1),) the probability of getting the desired orderings in step 4.(j) is $\theta_k\left(\frac{\ell_{2,j}}{n},\frac{\ell_{1,j+1}}{n}\right)+O_{k,a}\left(n^{-0.5}\right)$. The probability that we get the desired ordering in all steps $4.(1)$ up to $4.(a)$ is then
\[\theta_k\left(\frac{\ell_{2,1}}{n},\frac{\ell_{1,2}}{n}\right)\cdot\theta_k\left(\frac{\ell_{2,2}}{n},\frac{\ell_{1,3}}{n}\right)\cdot \ldots \cdot \theta_k\left(\frac{\ell_{2,a}}{n},\frac{\ell_{1,a+1}}{n}\right)+O_{k,a}\left(n^{-0.5}\right).\]

As noted before, there is a probability of $1-O_{k,a}\left(n^{-1/6}\right)$ that in step 3, all $\frac{\ell_{u,j}}{n}$ are at most $n^{-1/6}$ away from $\frac{m_{u,j}}{t_j}$. Hence, with probability $1-O_{k,a}$, after choosing $m_{u,j}$, the probability that we get the desired ordering in all steps 4.(1) up to 4.(a) is
\[\theta_k\left(\frac{m_{2,1}}{t_1},\frac{m_{1,2}}{t_2}\right)\cdot\theta_k\left(\frac{m_{2,2}}{t_2},\frac{m_{1,3}}{t_3}\right)\cdot \ldots \cdot \theta_k\left(\frac{m_{2,a}}{t_a},\frac{m_{1,a+1}}{t_{a+1}}\right)+O_{k,a}\left(n^{-1/6}\right),\]

where we have again used the fact that $\theta_k$ has bounded derivative (together with the fact that any input is changed by at most $n^{-1/6}$. To find the overall probability, it remains to sum over sequences \[m_{1,1},m_{2,1},m_{1,2},m_{2,2},m_{1,3},m_{2,3},\ldots, m_{1,a+1},m_{2,a+1}\]
such that for all $j\in [a+1]$,
\[m_{1,j}\neq m_{2,j} \hspace{3mm}\text{and}\hspace{3mm}1\leq m_{1,j},m_{2,j}\leq t_j.\]

Namely, $\mathbb{P}\left(\mathcal{E}_{I,n}\right)$ is the sum over all such sequences of the probability that we get this sequence in step 2 times the probability of getting the right ordering in step $4$ conditional on having this sequence in step 2. We have already found both of these probabilities. Writing it out, we get 
\[\mathbb{P}\left(\mathcal{E}_{I,n}\right)=\]\[\sum_{\substack{(m_{u,j})_{(u,j)\in [2]\times [a]}\\ \forall j, m_{1,j}\neq m_{2,j}\\ \forall j, 1\leq m_{1,j},m_{2,j}\leq t_j}}\frac{\varphi_k\left(\frac{m_{1,1}}{t_1}\right)\varphi_k\left(1-\frac{m_{2,1}}{t_1}\right)\ldots\varphi_k\left(\frac{m_{1,a+1}}{t_{a+1}}\right)\varphi_k\left(1-\frac{m_{2,a+1}}{t_{a+1}}\right)}{t_1^2\cdots t_{a+1}^2}\left(1+O\left(n^{-0.49}\right)\right)\]
\[\times \left(O_{k,a}\left(n^{-1/6}\right)+\left(1-O_{k,a}\left(n^{-1/6}\right)\right)\left(\theta_k\left(\frac{m_{2,1}}{t_1},\frac{m_{1,2}}{t_2}\right)\cdot \ldots \cdot \theta_k\left(\frac{m_{2,a}}{t_a},\frac{m_{1,a+1}}{t_{a+1}}\right)+O_{k,a}\left(n^{-1/6}\right)\right)\right),\]

The first $O_{k,a}\left(n^{-1/6}\right)$ term on the second line of this equation is the contribution of the choices in step $3$ where the gap between some $\frac{\ell_{u,j}}{n}$ and $\frac{m_{u,j}}{t_j}$ is more than $n^{-1/6}$. The $\left(1-O_{k,a}\left(n^{-1/6}\right)\right)$ term is the probability that all gaps turn out to be $\left\lvert\frac{\ell_{u,j}}{n}-\frac{m_{u,j}}{t_j}\right\rvert\leq n^{-1/6}$ in step $3$.

We can now collect all error terms into one additive error term of $O\left(n^{-1/6}\right)$ at the front, getting
\[\mathbb{P}\left(\mathcal{E}_{I,n}\right)=O_{k,a}\left(n^{-1/6}\right)\]\[+\sum_{\substack{(m_{u,j})_{(u,j)\in [2]\times [a]}\\ \forall j, m_{1,j}\neq m_{2,j}\\ \forall j, 1\leq m_{1,j},m_{2,j}\leq t_j}}\frac{\varphi_k\left(\frac{m_{1,1}}{t_1}\right)\varphi_k\left(1-\frac{m_{2,1}}{t_1}\right)\cdots\varphi_k\left(\frac{m_{1,a+1}}{t_{a+1}}\right)\varphi_k\left(1-\frac{m_{2,a+1}}{t_{a+1}}\right)}{t_1^2\cdots t_{a+1}^2}\]
\[\times\theta_k\left(\frac{m_{2,1}}{t_1},\frac{m_{1,2}}{t_2}\right)\cdot \ldots \cdot \theta_k\left(\frac{m_{2,a}}{t_a},\frac{m_{1,a+1}}{t_{a+1}}\right).\]

Now consider the following integral:
\[\int_{[0,1]^{2(a+1)}}\varphi_k(x_1)\varphi_k(1-y_1)\theta_k(y_1,x_2)\varphi_k(x_2)\varphi_k(1-y_2)\theta_k(y_2,x_3)\cdots\varphi_k(1-y_{a+1}).\]
Note that the sum appearing in the expression for $\mathbb{P}\left(\mathcal{E}_{I,n}\right)$ is a Riemann sum for this integral with cells of shape $\frac{1}{t_1}\times \frac{1}{t_1}\times \frac{1}{t_2}\times \frac{1}{t_2}\times\cdots \times \frac{1}{t_{a+1}}\times\frac{1}{t_{a+1}}$ minus the terms coming from cells corresponding to $m_{1,j}=m_{2,j}$. The domain $[0,1]^{2(a+1)}$ is compact so the integrand is bounded. Since all $t_j\geq \sqrt{n}$, the missing cells have total area at most $a\frac{1}{\sqrt{n}}=O_{k,a}\left(n^{-0.5}\right)$. The last two sentences together imply that the contribution of missing terms is at most $O_{k,a}\left(n^{-0.5}\right)$. As for the difference between the Riemann sum and the integral, since the derivative of the integrand (at any point and in any direction) is bounded, the error from each cell is at most its volume times $O_{k,a}\left(n^{-0.5}\right)$, for a total error of $O_{k,a}\left(n^{-0.5}\right)$. Hence, the difference between the sum in the expression for $\mathbb{P}\left(\mathcal{E}_{I,n}\right)$ and the integral is at most $O_{k,a}\left(n^{-0.5}\right)+O_{k,a}\left(n^{-0.5}\right)=O_{k,a}\left(n^{-0.5}\right)$. Plugging this integral into the expression for $\mathbb{P}\left(\mathcal{E}_{I,n}\right)$, we arrive at

\[\mathbb{P}(\mathcal{E}_{I,n})=O_{k,a}\left(n^{-1/6}\right)+\int_{[0,1]^{2(a+1)}}\varphi_k(x_1)\varphi_k(1-y_1)\theta_k(y_1,x_2)\varphi_k(x_2)\varphi_k(1-y_2)\theta_k(y_2,x_3)\cdots\varphi_k(1-y_{a+1}).\]

Crucially, this last integral only depends on $k$ and $a$ (and not $I$ or $n$). We can make the error term multiplicative by noting that this constant $C_{k,a}$ is greater than $0$. This was what remained to be proven.
\end{proof}

\section{Open problems}
One open problem is whether Conjecture~\ref{conj:dudu} can be proven for any $i$ (but asymptotically in $n$) with our approach. This would probably require a better understanding of the constants $c_{I,k}$. Another direction would be to prove versions of Theorem~\ref{thm:fmnasy}, Theorem~\ref{thm:dasy}, Theorem~\ref{thm:joint}, Theorem~\ref{thm:equistrong} with smaller error bounds. For Theorem~\ref{thm:equistrong}, one can also investigate the range of gap sizes between descents (in place of $\sqrt{n}$) for which an analog of the result holds, as well as try to find a good dependence of the error term on the gap size. Two more open problem are whether the discussion in Subsection~\ref{subsec:heuristic} can be made rigorous, and whether there is a way to simplify or explicitly understand the expression for $T_k(x,y)$ in Proposition~\ref{prop:genf}.

It also remains open whether this method can be generalized to other consecutive patterns in place of $k, k-1, \ldots, 1$. We expect that with an approach like the one in this paper, the main difficulty is in obtaining a nice expression for the analog of $\varphi_k(x)$, as we doubt that something as nice as Proposition~\ref{prop:fmn} will be available. However, we think that analogs of Theorem~\ref{thm:dasy} and Theorem~\ref{thm:equi} should still hold for other consecutive patterns. Before stating these as conjectures, we define some notation for other consecutive patterns. For $\pi\in S_k$ and $w\in S_n$, we let $P_\pi(w)$ be the set of starting indices of consecutive patterns $\pi$ in $w$. 

\begin{definition}
For $n\in \mathbb{Z}^+$ and $I\subseteq \mathbb{Z}^+$ a finite set, we let
\[\mathcal{P}_{\pi}(I,n)=\{w\in S_n\colon P_\pi(w)=I\}\hspace{5mm}\text{and}\hspace{5mm} p_\pi(I,n)=|\mathcal{P}_\pi(I,n)|.\]
We call $\mathcal{P}_\pi(I,n)$ the consecutive-$\pi$-function. Furthermore, we let
\[p_\pi(n):=p_{\pi}(\emptyset,n).\]
\end{definition}

To connect this up with familiar notation, note that $D_k(w)=P_{k,\ldots,1}(w)$, $\mathcal{D}_k(I,n)=\mathcal{P}_{k,\ldots,1}(I,n)$, and $f_k(n)=p_{k,\ldots,1}(n)$. We conjecture the following analogue of Theorem~\ref{thm:dasy}.

\begin{conjecture}
For any $k\geq 3$, any $\pi\in S_k$, and any finite $I\subseteq \mathbb{Z}^+$, there are constants $c_{\pi,I},\alpha\in \mathbb{R}$ with $\alpha>0$ such that 
\[p_{\pi}(I,n)=c_{\pi,I}p_\pi(n)\left(1+O\left(n^{-\alpha}\right)\right).\]
\end{conjecture}

We conjecture the following analogue of Theorem~\ref{thm:equistrong}.

\begin{conjecture}
Fix $k\geq 3$, $\pi\in S_k$, and $a\in \mathbb{Z}^+$. There is a real constant $C_{\pi,a}>0$ such that for $n\in \mathbb{Z}^+$ and $I=\{i_1,i_2,\ldots, i_a\}\subseteq [n]$ with $i_1\geq \sqrt{n}$, $i_{j+1}-i_j\geq \sqrt{n}$ for all $j$ with $1\leq j\leq a-1$, and $n-i_a\geq \sqrt{n}$, we have
\[p_\pi(I,n)=C_{k,a} p_\pi(n)\left(1+O_{\pi,a}\left(n^{-\alpha}\right)\right).\]
\end{conjecture}

The following random process also seems interesting. We construct a random permutation that avoids $k$-descents by inserting last elements one at a time (with value between two previous values, chosen uniformly among options that avoid $k$-descents). In this language, Theorem~\ref{thm:fmnasy} says that the distribution of the last element is given by $\varphi_k$, and Theorem~\ref{thm:joint} says that the correlation between the value of the first and last element is asymptotically small. One can ask about other properties of this random process, such as whether any two points are asymptotically uncorrelated, and also investigate the formation of other patterns.

\section{Acknowledgments}
This research was carried out under the MIT Math Department 2020 UROP+ summer research program. My mentor was Pakawut Jiradilok, whom I would like to thank for excellent mentorship and very many great ideas, as well as for his patience and understanding. I would also like to thank Wijit Yangjit for a helpful conversation on asymptotics of power series coefficients; Slava Gerovitch, David Jerison, and Ankur Moitra for organizing the UROP+ program; and the Meryl and Stewart Robertson UROP Fund for funding this project.

\bibliographystyle{plain}
\bibliography{ref}

\begin{thebibliography}{10}

\bibitem{billey}
Sara Billey, Krzysztof Burdzy, and Bruce~E. Sagan.
\newblock Permutations with given peak set.
\newblock {\em J. Integer Seq.}, 16(6):Article 13.6.1, 18, 2013.

\bibitem{MR3463566}
Sara Billey, Matthew Fahrbach, and Alan Talmage.
\newblock Coefficients and roots of peak polynomials.
\newblock {\em Exp. Math.}, 25(2):165--175, 2016.

\bibitem{DAVIS20183249}
Robert Davis, Sarah~A. Nelson, T.~{Kyle Petersen}, and Bridget~E. Tenner.
\newblock The pinnacle set of a permutation.
\newblock {\em Discrete Mathematics}, 341(11):3249 -- 3270, 2018.

\bibitem{DIAZLOPEZ201721}
Alexander Diaz-Lopez, Pamela~E. Harris, Erik Insko, and Mohamed Omar.
\newblock A proof of the peak polynomial positivity conjecture.
\newblock {\em Journal of Combinatorial Theory, Series A}, 149:21 -- 29, 2017.

\bibitem{diaz}
Alexander Diaz-Lopez, Pamela~E. Harris, Erik Insko, Mohamed Omar, and Bruce~E.
  Sagan.
\newblock Descent polynomials.
\newblock {\em Discrete Math.}, 342(6):1674--1686, 2019.

\bibitem{kitaev}
Richard Ehrenborg, Sergey Kitaev, and Peter Perry.
\newblock A spectral approach to consecutive pattern-avoiding permutations.
\newblock {\em J. Comb.}, 2(3):305--353, 2011.

\bibitem{ELIZALDE2003110}
Sergi Elizalde and Marc Noy.
\newblock Consecutive patterns in permutations.
\newblock {\em Advances in Applied Mathematics}, 30(1):110 -- 125, 2003.

\bibitem{flajolet}
Philippe Flajolet and Robert Sedgewick.
\newblock {\em Analytic Combinatorics}.
\newblock Cambridge University Press, USA, 1 edition, 2009.

\bibitem{gaetz2019qanalogs}
Christian Gaetz and Yibo Gao.
\newblock On $q$-analogs of descent and peak polynomials.
\newblock {\em arXiv preprint arXiv:1912.04933}, 2019.

\bibitem{jiradilok2019roots}
Pakawut Jiradilok and Thomas McConville.
\newblock Roots of descent polynomials and an algebraic inequality on hook
  lengths.
\newblock {\em arXiv preprint arXiv:1910.14631}, 2019.

\bibitem{macmahon}
Percy~A. MacMahon.
\newblock {\em Combinatory analysis. {V}ol. {I}, {II} (bound in one volume)}.
\newblock Dover Phoenix Editions. Dover Publications, Inc., Mineola, NY, 2004.
\newblock Reprint of {{\i}t An introduction to combinatory analysis} (1920) and
  {{\i}t Combinatory analysis. Vol. I, II} (1915, 1916).

\bibitem{Oguz2019DescentPP}
Ezgi~Kantarci O\u{g}uz.
\newblock Descent polynomials, peak polynomials and an involution on
  permutations.
\newblock 2019.

\bibitem{oeis}
Neil J.~A. Sloane et~al.
\newblock The on-line encyclopedia of integer sequences, 2020.

\bibitem{10.5555/2124415}
Richard~P. Stanley.
\newblock {\em Enumerative Combinatorics: Volume 1}.
\newblock Cambridge University Press, USA, 2nd edition, 2011.

\bibitem{warlimont}
R.~Warlimont.
\newblock Permutations avoiding consecutive patterns.
\newblock {\em Ann. Univ. Sci. Budapest. Sect. Comput.}, 22:373--393, 2003.

\bibitem{zhu2019enumerating}
Christopher Zhu.
\newblock Enumerating permutations and rim hooks characterized by double
  descent sets.
\newblock {\em arXiv preprint arXiv:1910.12818}, 2019.

\end{thebibliography}

\end{document}